\documentclass[a4paper,11pt,oneside,reqno]{amsart}
\usepackage[T1]{fontenc}

\usepackage{amsfonts,amsmath, amssymb,amsthm,amscd}
\usepackage{bm}
\usepackage{mathrsfs}
\usepackage{yfonts}
\usepackage{verbatim}
\usepackage{hyperref}
\usepackage{graphicx}
\usepackage[utf8]{inputenc}
\usepackage{latexsym}
\usepackage{upgreek}
\usepackage{relsize}
\usepackage{xcolor}
\usepackage{mathdots}
\usepackage{mathtools}
\usepackage{dsfont}
\usepackage[left=3cm, right=3cm, top=3cm, bottom=3cm]{geometry}
\usepackage{tikz}
\usepackage{lmodern}
\usepackage{stmaryrd}
\usepackage{enumitem}
\usepackage{braket}

\usepackage[backend=biber,style=alphabetic, maxnames=5, maxalphanames=5, sortcites=false, giveninits=true, isbn=false, clearlang=true, doi=false, url=true, sorting=nyt]{biblatex}
\renewbibmacro{in:}{}
\AtEveryBibitem{\clearlist{language}} 
\AtEveryBibitem{\clearfield{eprintclass}}

\numberwithin{equation}{section}

\usepackage[width=.9\textwidth]{caption}

\makeatletter
\renewcommand\section{\@startsection{section}{1}%
		\z@{1.5\linespacing\@plus.5\linespacing}{1\linespacing}%
		{\large\normalfont\scshape\centering}}
\makeatother

\newcommand{\EE}{\ensuremath{\mathbb{E}}}

\newcommand{\PP}{\ensuremath{\mathbb{P}}}
\newcommand{\R}{\ensuremath{\mathbb{R}}}
\newcommand{\C}{\ensuremath{\mathbb{C}}}
\newcommand{\Z}{\ensuremath{\mathbb{Z}}}

\newcommand{\I}{\ensuremath{\mathbf{i}}}

\renewcommand{\rho}{\varrho}

\newcommand{\eps}{\varepsilon}
\renewcommand{\leq}{\leqslant}
\renewcommand{\geq}{\geqslant}
\newcommand{\Tr}{\mathrm{Tr}}
\newcommand{\tr}{\mathrm{tr}}

\newcommand{\Kclassical}{\mathbf{K}}
\usetikzlibrary{patterns}
\usetikzlibrary{shapes.multipart}
\usetikzlibrary{arrows}

\newtheorem{theorem}{Theorem}[section]

\newtheorem{claim}[theorem]{Claim}
\newtheorem{assumption}{Assumptions}

\theoremstyle{definition}
\newtheorem{remark}[theorem]{Remark}

\theoremstyle{definition}
\newtheorem{example}[theorem]{Example}

\theoremstyle{definition}

\theoremstyle{definition}

\begin{document}	
	
	\title{Introduction to quantum exclusion processes}
	
	\author[G.~Barraquand]{Guillaume Barraquand}
	\address{G. Barraquand, Laboratoire de Physique de l'Ecole Normale Supérieure, Ecole Normale Supérieure, PSL University, CNRS, Sorbonne Université, Université Paris-Cité, 24 rue Lhomond, 75005 PARIS}
	\email{guillaume.barraquand@ens.fr} 
	\author[D.~Bernard]{Denis Bernard}
	\address{D. Bernard, Laboratoire de Physique de l'Ecole Normale Supérieure, Ecole Normale Supérieure, PSL University, CNRS, Sorbonne Université, Université Paris-Cité, 24 rue Lhomond, 75005 PARIS.}
	\email{denis.bernard@ens.fr}
	
	\begin{abstract}
The QSSEP, short for quantum symmetric simple exclusion process, is a paradigm model for stochastic quantum  dynamics. Averaging over the noise, the quantum dynamics reduce to the well-studied SSEP (symmetric simple exclusion process). These  notes provide an introduction to quantum exclusion processes, focusing on the example of QSSEP and its connection to free probability, with an emphasis on mathematical aspects. 
	\end{abstract}
	
	\maketitle
	\setcounter{tocdepth}{1}
	\tableofcontents
	
	\section{Introduction} \label{sec:Introduction}

    	The initial motivation for introducing quantum exclusion processes was to develop a quantum extension of the macroscopic fluctuation theory (MFT), which is a general framework to deal with diffusive systems driven out of equilibrium introduced in \cite{bertini2001fluctuations, bertini2005current}. Beyond this initial motivation, quantum exclusion processes are a good playground to extend the tools used in the large scale analysis of classical stochastic systems. In these notes, we  mostly discuss the simplest model, the  Quantum Symmetric Simple  Exclusion Process (QSSEP). The QSSEP is a noisy and quantum analogue of the Symmetric Simple Exclusion Process (SSEP). The average over the noise of QSSEP reduces to its classical counterpart,  the  SSEP.  As we will see, its analysis displays clear signs of integrability which deserve to be better understood, and is closely connected to the theory of random matrices and  free probability \cite{bernard2019open,biane2021combinatorics,hruza2022coherent}. Actually, free probability also appears in other studies of quantum systems, notably in the context of the eigenstate thermalization hypothesis \cite{pappalardi2022eigenstate}.

	\subsection{Towards a quantum extension of Macroscopic Fluctuation Theory}
	Let us elaborate on the initial motivation of finding a quantum extension to MFT \cite{BauerBernardJin17}. Consider a stochastic particle system on a lattice, let's say a Markovian exclusion process  on $N$ sites between two reservoirs at densities $n_a\neq n_b$. At large scale, the system is often  described by a density function $n(x, s)$ where the variables $x\in [0,1]$ and $s>0$ denote a macroscopic position $x$ and rescaled time $s$: $n(x,s)$ is the average number of particles per site near position $\lfloor xN\rfloor$ at time $s N^2$. Such a hydrodynamic limit is  generally based on the fact that locally, the system converges to a stationary measure corresponding to a certain density, which may however vary at large scale. This can be completely formalized mathematically \cite{kipnis1998scaling}. If $j(x,s)$ denotes the average flux near the macroscopic position $x$, the conservation of particles imposes that  
	\begin{equation}
		\partial_s n(x,s) +\partial_x j(x,s) =0.
		\label{eq:conservation}
	\end{equation} 
	Macroscopic fluctuation theory allows to predict large deviation principles for various quantities related to transport in  diffusive systems. The main idea is to introduce a noisy version of the Fourier-Fick's law
	\begin{equation}
	 j(x,s) = -D(n(x,s))\partial_x n(x,s) + \sqrt{\frac{1}{N}} \sqrt{\sigma(n(x,s))}\, \xi(x,s),
	\label{eq:weaknoise}
	\end{equation}
	 where $\xi$ is a space-time white noise, which induces a probability measure on density and current profiles. 
	  Here $D(n)$ is the diffusion constant and $\sigma(n)$  is called the mobility; we do not define them here.
	 Since the noise variance is of order $\frac{1}{N}\ll 1$, density and current fluctuations are small and rare. This framework allows to find a functional $\mathcal I$, called rate function, sometimes completely explicit,  such that under the stationary measure of the process,  
	$$  \mathbf P\left[  (n(x))_{x\in [0,1]} \approx f\right] \approx e^{-N \mathcal I(f)},$$
to be understood using the formalism  of large deviation theory --- we refer to  \cite{bertini2014macroscopic} for a review on MFT. From the mathematical point of view, the validity of MFT in general is only conjectural, but its predictions have been rigorously verified  on many specific models (including SSEP). 

	A quantum extension of MFT  for diffusive quantum systems   should describe: 
	\begin{enumerate}
		\item transport and their fluctuations;
		\item quantum effects (interferences, coherences) and their fluctuations.
	\end{enumerate}
We refer to the review  \cite{bernard2021can} for further physical background. The first task is to define quantum extensions of interacting particle systems. In these notes we focus on exclusion type systems, and often restrict to the guiding example of the symmetric simple exclusion process. Before defining it in Section \ref{sec:QSSEP}, we need to explain what is the state space in the quantum setting,  and how dynamics are defined. 

\subsection{The state space for quantum systems}
The particle configuration state space $\mathcal C$ of the classical system becomes a $*$-algebra $\mathcal A$. The reader unfamiliar with this formalism  may  think of $\mathcal A$ as a set of linear operators, acting on some Hilbert space, and their adjoints. A probability measure on $\mathcal C$ becomes a \emph{state} on $\mathcal A$, that is a positive unital linear functional $\nu: \mathcal A\to \mathbb C$ satisfying $\nu(a^*) =\overline{\nu(a)}$, $\nu(1)=1$, $\nu(a^*a)\geq 0$. Typically, the algebra $\mathcal A$ is the von Neumann algebra  $\mathcal A=\mathcal B(\mathcal H)$  of bounded linear operators on some Hilbert space  $\mathcal H$. When $\dim \mathcal H$ is finite, states may be encoded by a density matrix $\rho$ (positive semi-definite Hermitian) such that 
	$$ \nu(a) = \mathrm{Tr}(\rho a).$$
	
\begin{example} For the SSEP on $N$ sites, we denote configurations by   $\mathfrak{n}=(n_1, \dots, n_N) \in \mathcal C=\lbrace 0,1\rbrace^N$. Let $\mathcal H$ be a Hilbert space with basis $(\ket{\mathfrak{n}})_{\mathfrak{n}\in \mathcal C}$ indexed by classical configurations.  Here we use Dirac's notation: a so-called ket $\ket{u}$ denotes a vector in some Hilbert space. A braket $\braket{v\vert u}$ is a complex number, the scalar product of $\ket u$ and $\ket v$. Hence we may view the bra $\bra v$ as a linear form $\mathcal H\to\mathbb C$. When $\mathcal H$ is finite dimensional, one should think of $\bra v$ as being the  conjugate transpose of $\ket v$, i.e. a row vector. Finally,  $\ket u \bra v$ is a rank one linear operator on the Hilbert space.  Since $\dim(\mathcal H) = 2^{N}$,  we can focus on the $2^N\times 2^N$  density matrix $\rho$. Note that if $\rho$ is diagonal in the basis $(\ket{\mathfrak{n}})_{\mathfrak n\in \mathcal C}$, it can be written as 
$$\rho = \sum_{\mathfrak{n}\in \mathcal{C}}  \mathbf{P}(\mathfrak{n}) \ket{\mathfrak{n}}\bra{\mathfrak{n}}, $$
where $\mathbf P(\mathfrak{n})$ is a probability distribution. Indeed, the condition $\nu(1)=1$ is equivalent to $\mathrm{Tr}(\rho)=1$, that is $\sum_{\mathfrak{n}\in \mathcal C} \mathbf P(\mathfrak{n}) =1$, and positivity implies $0\leq \mathbf P(\mathfrak{n})\leq 1$.
\end{example}
\subsection{Stochastic quantum dynamics} 
We also need to define dynamics on states, instead of Markovian dynamics on probability measures.  In general, quantum dynamics on quantum states can either be unitary flows for closed systems, or flows of so-called trace preserving completely positive maps (also called CP-maps or quantum channels) for open systems, since then the dynamics are in general not unitary to take into account additional dissipative processes related to the contacts with external reservoirs.
Here, these dynamics will be themselves random, hence we work on a probability space $(\Omega, \mathcal F, \PP)$.  We wish to define a random state  $\nu_t^\omega$ for each $\omega\in \Omega$ (or a random density matrix  $\rho_t^\omega$). 
 The dynamics are constructed using random unitary flows for closed systems, and flows of random CP-maps for open systems.

 We moreover impose that our   quantum stochastic dynamics
should be related to the classical one, in the sense that the average quantum dynamics 
$$ t \mapsto \overline{\rho_t} := \mathbb E\left[ \rho_t\right] $$
should contain, as a marginal, the classical dynamics.

 We emphasize that there are two levels of randomness when dealing with quantum stochastic dynamics:
	\begin{itemize}[leftmargin=1em]
		\item Quantum randomness,  encoded by $\nu$ (or $\rho$). The quantum expectation value is $\nu(a)$. 
		\item Classical randomness, coming from the fact that $\nu$ depends on $\omega\in\Omega$, associated to the probability measure $\PP$ and expectation $\EE$ -- note that we use a different notation than for the probability measure $\mathbf P$ and expectation $\mathbf E$ associated to the classical SSEP. 
	\end{itemize}
	Thus we should consider quantities such as 
	$$  \EE\left[\nu_t(a)\right], \;\;\EE\left[\nu_t(a_1)\nu_t(a_2)\right], \cdots, \;\;\EE\left[\nu_t(a_1)\dots \nu_t(a_n)\right].$$ 
	 Furthermore, it may also be interesting to see how the dynamics may be influenced by a sequence of quantum measurements, but we will not discuss this here. 

\begin{remark}
From a physics perspective, the origin of randomness may be understood from two different perspectives:
\begin{itemize}[leftmargin=1em]
\item Either as a way to describe  generic systems;
\item Or as arising from interaction with hidden degrees of freedom.  
\end{itemize}
The latter case is particularly adapted to describe diffusive systems. Indeed a physical particle system behaves diffusively only above a certain scale, the so-called \emph{mean free path} which is the distance separating two collisions of the system's particles on obstacles (think about the passage from random walks to the Brownian motion).  Thus, if we are aiming at describing diffusive systems, there should be low space-time scale degrees of freedom inducing these many collisions. Furthermore, after averaging over those degrees of freedom, the system is expected to behave classically -- in the physics literature, this phenomenon of destructive interference induced by collisions is called decoherence. This echoes in the above condition that the average dynamics should reduce to the classical one.
\end{remark}
	
\begin{remark} For the quantum SSEP, the density matrix $\rho_t$ is a $2^N\times 2^N$ matrix. As we have mentioned, when it is diagonal, it describes a probability measure on configurations, but in the quantum setting we may also look at combinations of classical configurations such as 
	$$ \ket{010011100}+\ket{010001100} ,$$
and study interferences. Hence the off-diagonal nature of $\rho_t$ is important and codes for quantum effects. 
\end{remark}
	
	\subsection{Quantum exclusion processes}
In the case of quantum exclusion processes on $N$ sites with periodic boundary conditions, the system is closed and the dynamics are unitary, in the sense that 
	$$	 \rho_t = U_t \rho_0 U_t ^* ,$$
	where $U_t=U_t(\omega)$, $\omega\in\Omega$, is a unitary matrix described by a stochastic differential equation. We consider It\={o} type SDEs of the form
	$$ \rho_{t+dt} =e^{-\I dH_t} \rho_t e^{\I dH_t},$$
	for some Hermitian matrix Hamiltonian   increments $dH_t$. By definition, the average over the noise  $\overline{\rho_t}$ is 
	$$ \overline{\rho_t}  = \int_{\Omega} d\PP(\omega) U_t(\omega) \rho_0 U_t(\omega)^* =: \overline \Phi_t(\rho_0).$$
This is a convex combination of unitary transformations, but it is not unitary in general.  The map $\rho\mapsto \overline \Phi_t(\rho)$ is an example of a (special class of) trace preserving completely positive map, 
i.e. $\Tr(\overline \Phi(\rho))=1$ for $\Tr(\rho)=1$, and $\overline \Phi(\rho)\geq 0$ for $\rho\geq 0$. Technically, the definition of CP-maps involves a stronger notion than the simple positivity, but we don't need to enter into these details --  see \cite[chapter 6]{attal2025+lectures} for the more complete definition. 

In the case of open quantum exclusion processes, the above unitary flows are replaced by flows of CP-maps $\rho_0 \mapsto \rho_t=\Phi_t(\rho_0)$, with $\Phi_t=\Phi_t(\omega)$, $\omega\in\Omega$. The convex combination of CP-maps is again a CP-map so that the average flow is also a CP-map.

As already mentioned, we require that the semigroup $\overline \Phi_t$ encodes the classical dynamics. This means that  $\overline \Phi_t$ should preserve the class of diagonal matrices (in the basis $(\ket{\mathfrak{n}})_{\mathfrak n\in \mathcal C}$) and the associated flow on diagonal matrices should be the same as the evolution of probability distributions $\mathbf P_t(\mathfrak n)$ in the classical exclusion process. 

For instance, in the special case of the SSEP, it implies that 
\begin{equation}
	\mathbf E_{\rm ssep}\left[ e^{\sum_{j=1}^N h_j n_j(t)} \right] = \mathbb E\left[ \Tr\left(  \rho_t e^{\sum_{j=1}^N h_j \hat n_j} \right) \right],  
	\label{eq:functionalQSSEPSSEP}
\end{equation} 
	where $\hat n_j$ denotes the number operator, i.e. the diagonal matrix with diagonal coefficients $0$ or $1$ (in the basis $(\ket{\mathfrak{n}})_{\mathfrak n\in \mathcal C}$) depending on the state of the $j$th site.	 
	
	\subsection{A curiosity about the classical SSEP}
	As $t$ goes to infinity, the average density matrix $\overline{\rho_t}$ must converge to the unique stationary measure of the SSEP. Under this stationary measure, it is well-known that, 
	$$ 
	\mathbf E_{\rm ssep}\left[ e^{\sum_{j=1}^N h_j n_j} \right]  \asymp_{N\to\infty} e^{N \mathcal F_\mathrm{ssep}(h)},
	$$ 
	in the sense that the logarithm of both sides are asymptotically equivalent,  
	where the function $h$ is a  continuum limit of $h_j$ in a suitable sense  (say $h_j=h(j/N)$ for some smooth function $h$). The statement was first derived in  \cite{derrida2001free}, and later with various levels of mathematical rigor in \cite{derrida2001large, bertini2003large, bertini2007stochastic, bodineau2022large}. Via \eqref{eq:functionalQSSEPSSEP}, the analysis of  QSSEP   led \cite{bauer2022bernoulli}  to a new expression for the SSEP large deviation rate function  (here we consider the system with boundary densities $n_a=0, n_b=1$):
	\begin{equation}
		\mathcal F_\mathrm{ssep}(h) = \sup_{a, b}\left\lbrace \int_0^1 dx \left[\log\left( 1+ b(x)(e^{h(x)}-1)-a(x)b(x)\right)\right] +\tilde{ \mathcal F}_0(a) \right\rbrace,
		 \label{eq:variationalpb}
	\end{equation}
	where 
	$$ \tilde{\mathcal F}_0(a) = \sum_{k=1}^{\infty} \frac{(-1)^{k+1}}{k} \kappa_k\left( \int_x^1 \mathrm dy\, a(y) \right),$$
	and $\kappa_k$ denotes the $k$th free cumulant defined in Section \ref{sec:cumulants}. The map  $x\mapsto \int_x^1 \mathrm dy\, a(y)$ should be viewed as a random variable on the space $[0,1]$ equipped with the Lebesgue measure.  The variational problem in \eqref{eq:variationalpb} does not involve any boundary condition on $a,b$ and should be understood in a sense explained below in Section \ref{sec:structured}. 

\begin{remark}	\label{remark:classicality}
	Further, the analysis of QSSEP suggests that, for a generic quantum diffusive system (at large time),  the large deviation fluctuations of quantum observables related to density or transport are those of its classical system counterpart. 
\begin{itemize}[leftmargin=1em]
	\item
 In particular, it has been conjectured that, under mild hypotheses, we have  the $\PP$-almost sure convergence 
	$$ \frac{1}{N} \log \Tr\left(\rho_{\rm eq} e^{\sum_{j=1}^N h_j \hat n_j} \right)  \xrightarrow[N\to\infty]{} \mathcal F_Q(h),$$
where $\rho_{\rm eq}$ denotes a random density matrix whose probability distribution is left invariant by the stochastic quantum dynamics and  the functional $\mathcal F_Q$ is deterministic, the same as for the classical system i.e. $\mathcal F_{Q}=\mathcal F_\mathrm{classical}$. However,  sub-leading corrections (in $1/\sqrt{N}$) are expected to be fluctuating and different in the classical and quantum processes. A similar conjecture  is expected to hold for the large deviation functions for transport quantities (e.g. current) \cite{hruza2022coherent,McCulloch2023}. 
\item Furthermore, it has also been conjectured in \cite{hruza2022coherent} that, if the macroscopic fluctuation theory applies to density fluctuations, then the coherent off-diagonal matrix elements of $\rho_\mathrm{eq}$ should obey a large deviation principle (with free probability structures).
\end{itemize}
\end{remark}

\subsection{Status of the results} 
Throughout the text we use the following convention. Theorems are statements for which a complete proof exists in the literature (a reference is given). Claims are statements that we regard as true and for which a derivation exists at the level of rigor of theoretical physics, but for which no complete mathematical proof is available. 
A list of mathematical open problems is given in Section \ref{sec:open}.

\section{Quantum Symmetric Simple Exclusion Process}
\label{sec:QSSEP}

Like its classical counterpart, QSSEP can be defined on any graph. Here we will focus on finite graphs: 
\begin{enumerate}
	\item A chain of $N$ sites indexed $1, \dots, N$ with periodic boundary conditions; 
	\item A closed chain of $N$ sites without imposing periodic boundary conditions; 
	\item A chain of $N$ sites connected to boundary reservoirs, which we call open boundary conditions. 
\end{enumerate}
As we will see, the analysis of the first two cases is very similar, while the third case is much richer. Since we have $N$ sites in all cases,  the Hilbert space $ \mathcal H \cong (\mathbb{C}^2)^{\otimes N}$, and the system is described by a $2^N\times 2^N$ density matrix $\rho_t$. 

\subsection{Periodic boundary conditions}
 In that case, the system is not in contact with external reservoirs and forms a closed system. The evolution is a stochastic unitary dynamics. 
 
 \subsubsection{Definition of the model}
 The time evolution of the density matrix $\rho_t=\rho_t^\omega$ is defined by the matrix stochastic differential equation
\begin{equation}
	\rho_{t+dt} = e^{-\I \mathrm dH_t} \rho_t e^{\I \mathrm dH_t},
	\label{eq:SDE}
\end{equation} 
with Hamiltonian  increment
$$ \mathrm dH_t = \sum_{j=1}^{N} c_{j+1}^{\dagger} c_j \mathrm dW_t^{(j)}+ c_j^{\dagger} c_{j+1} \mathrm d\overline W_t^{(j)} ,$$
with indices modulo $N$ ($c_{N+1}=c_1$),
where   $(W_t^{(j)})_{1\leq j\leq N}$ are independent complex Brownian motions and $\overline W_t^{(j)}$ denote their complex conjugate, so that we have 

\begin{equation}
 \left\langle W^{(j)},  \overline W^{(k)} \right\rangle_t   = \delta_{jk} \,t, \;\;  \left\langle W^{(j)},   W^{(k)} \right\rangle_t   = 0, \;\;   \left\langle \overline W^{(j)},  \overline W^{(k)} \right\rangle_t= 0,
\label{eq:correlationsBrownians}
\end{equation}
where $\delta_{jk}$ is the Kronecker delta symbol,  equal to $1$ when $j=k$ and $0$ otherwise, and $\langle X,Y\rangle_t$ denotes the quadratic (co)variation of stochastic processes $X_t$ and $Y_t$. 
 There is one complex Brownian motion per edge.

 The $c_j^\dagger$ and $c_k$ are respectively fermionic creation and annihilation operators, satisfying anti-commutation relations 
$$ \lbrace c_j^\dagger,  c_k\rbrace = \delta_{jk}, \;\; \lbrace c_j^\dagger, c_k^\dagger\rbrace = \lbrace c_j, c_k\rbrace =0 ~. $$ 
Here $\lbrace a, b\rbrace$ is the anti-commutator $ab+ba$.  When $N=1$, $c$ and $c^\dagger$ are given by the $2\times 2$ matrices 
$$ c= \begin{pmatrix}
	0&1\\0&0
\end{pmatrix}, \;\;\; c^\dagger = \begin{pmatrix}
0&0 \\ 1&0
\end{pmatrix},$$
so that 
$c^\dagger \ket 0= \ket 1$ and $c\ket 1=\ket 0$ if $\ket 0 = \begin{pmatrix} 1,0 \end{pmatrix}^T$ and $\ket 1 = \begin{pmatrix} 0,1 \end{pmatrix}^T$. 
We remark that the symbol $\dagger$ denotes, in fact, the Hermitian conjugation, which we could denote by $*$. We keep using the symbol $\dagger$ as this is the most common when dealing with creation and annihilation operators. When $N>1$, one should define the action of operators $c_j$ and $c_k^\dagger$ on $\left( \mathbb C^2\right)^{\otimes N}$. Concretely, they have a block structure acting independently on each component of the tensor product: 
$$ c_k  = \underbrace{
	\begin{pmatrix}
	1&0\\0&-1
\end{pmatrix}  
\otimes \dots \otimes 
\begin{pmatrix}
1&0\\0&-1
\end{pmatrix}  
}_{k-1  \text{ times}} \otimes \, c \otimes \underbrace{I\otimes \dots \otimes I}_{N-k \text{ times}},$$
and $c_k^{\dagger}$ is the adjoint (Hermitian conjugate) -- see e.g. \cite[Appendix A]{stephan2020extreme}. We will see  another equivalent  definition in Section \ref{sec:Fock}  with a more transparent algebraic interpretation.  

By It\=o's formula,  \eqref{eq:SDE} can be rewritten as 
\begin{equation*}
	\mathrm d\rho_t = -\I [ \mathrm dH_t, \rho_t] -\frac 1 2 \underbrace{[\mathrm dH_t , [\mathrm dH_t, \rho_t]]}_{\text{It\={o} contraction}} ,
\end{equation*}
where the double commutator is computed by applying the rules \eqref{eq:correlationsBrownians}, i.e. all quadratic terms in $\mathrm dW^{(j)}_t$ or $\mathrm d\overline W^{(j)}_t$   are replaced by the differential of their quadratic (co)variation (these are the rules of It\={o}'s calculus). This description shows that $\rho_t$ solves a matrix valued  linear stochastic differential equation, hence it admits a unique strong solution.

\subsubsection{Average dynamics}
We can now easily  describe the average dynamics $\overline {\rho_t} = \mathbb E[\rho_t]$. We have 
\begin{equation} \label{eq:averaged-motion}
\mathrm d\overline{ \rho_t} = -\frac{1}{2} \mathbb E [\mathrm dH_t , [\mathrm dH_t, \overline{\rho_t}]] =: \overline{ \mathcal L} (\overline{\rho_t}) \, \mathrm dt ,
\end{equation}
where, again, the double commutator is simplified using  \eqref{eq:correlationsBrownians}. The linear operator $\overline{ \mathcal L} $ is called a \emph{Lindbladian}. This is the generator associated to the semigroup $\overline \Phi_t$ of trace preserving completely positive maps, 
$$ \overline {\rho_t}  = e^{t\overline{ \mathcal L} } \rho_0 = \overline \Phi_t(\rho_0). $$
Explicitly, the Lindbladian is given by 
$ \overline{ \mathcal L}   = \overrightarrow{\mathcal L} + \overleftarrow{\mathcal L}$
with 
$$ \overrightarrow{\mathcal L} \rho= \sum_{j=1}^{N} \ell_j \rho\ell_j^\dagger -\frac 1 2 \left( \ell_j^\dagger\ell_j \rho+\rho \ell_j^\dagger \ell_j \right),$$
where $\ell_j=c_{j+1}^\dagger c_j$. The operator $\overleftarrow{\mathcal L}$ is defined similarly, simply exchanging the roles of  $\ell_j$ and  $\ell_j^\dagger$. 
Both $\overrightarrow{\mathcal{L}}$ and  $\overleftarrow{\mathcal{L}}$ preserve density matrices $\rho_{\rm diag}$ diagonal in the particle number basis $\ket{\mathfrak{n}}$, $\mathfrak n\in\mathcal{C}$. 
 As mentioned above, such density matrix specifies a probability measure on classical configurations since it can be written as $\rho_\mathrm{diag}= \sum_{\mathfrak n\in \mathcal{C}}  \mathbf{P}(\mathfrak n) \ket{\mathfrak{n}}\bra{\mathfrak{n}} $.  To see how $\overline {\mathcal L}$ acts on $\rho_{\rm diag}$, it is enough to check how it acts on any edge configuration (since $\overline {\mathcal L}$ is the sum of terms acting on edges). We have:
\begin{align*}
\overline{ \mathcal L}  (\ket{00}\bra{00}) & =0 ,\\
\overline{ \mathcal L} (\ket{01}\bra{01}) & =  \ket{10}\bra{10} -\ket{01}\bra{01},\\
\overline{ \mathcal L} (\ket{10}\bra{10}) & =  \ket{01}\bra{01}-\ket{10}\bra{10} ,\\
\overline{ \mathcal L} (\ket{11}\bra{11}) & =0 .
\end{align*}
Thus $\overline{ \mathcal L}$ acts as the Markov generator of SSEP on diagonal density matrices.

\begin{remark}
Since $\rho_t$ is a random variable, one may also look at the evolution of its higher moments. Those are going to evolve linearly in time, as for the average \eqref{eq:averaged-motion}, but with a Lindblad operator depending on the order of the moment.
\end{remark}

\subsubsection{An important property} 
	The QSSEP dynamics possesses a $U(1)^N$ invariance. This means that we can change $c_j$ into $c_je^{\I\theta_j}$, $\theta_j\in\mathbb{R}$,  (so that $c_j^\dagger$ is changed into $c_j^\dagger e^{-\I\theta_j}$) without changing the law of the states $\rho_t$ because such transformation can be absorbed in a redefinition of the Brownian motions without changing their law.
	
\subsubsection{Fock space formalism and further remarks}
\label{sec:Fock}
	It may be convenient to work with the Fock space representation of the Hilbert space $\mathcal H$ spanned by $\ket{\mathfrak n}, \mathfrak n \in \mathcal C$, considered in the previous section. It is isomorphic to the exterior algebra $\bigwedge\! V$ where $V = \C^N$ is the one-particle Hilbert space, i.e. 
$$  \mathcal H \cong \bigwedge\! V = \C\oplus V\oplus V\wedge V \oplus V\wedge V\wedge V\oplus \dots $$ 
We then view $\C$ as spanned by zero particle configurations, $V$ spanned by  one particle configurations, $V\wedge V$ by two particle configurations, etc. 
In this representation, the creation and annihilation fermion operator act via wedge products or contractions on elements of the exterior algebra $\bigwedge \!V$. More precisely, if we denote by $(e_i)_{1\leq i\leq N}$ an orthonormal basis of $V$, we define $c_i^\dagger = c^+(e_i)$ and $c_i=c^-(e_i)$ where the linear operators $c^+(u)$ and $c^-(u)$ are defined for $u\in V$ as 
	\begin{align*}
		c^+(u) (u_1 \wedge \dots \wedge u_n) &= u\wedge u_1\wedge \dots \wedge u_n,\\
		c^-(u) (u_1 \wedge \dots \wedge u_n) &= \sum_{i=1}^n (-1)^{i-1} \langle u, u_i\rangle u_1\wedge \dots \wedge \widehat{u_i}\wedge \dots\wedge u_n,
	\end{align*}
 where $\widehat{u_i}$ indicates that the contribution of $u_i$ is absent in the wedge product.
	The local number operators $\hat n_j$ can be written as $\hat n_j= c_j^\dagger c_j$. They count the number of particles at site $j$. We refer to \cite{meyer1995quantum} for a more complete presentation.
	
	\begin{remark}
	By the conservation of particles, for each $1\leq k\leq N$, the subspace $V^{\wedge k}$ is stable under the dynamics, hence $U_t$ is a block matrix. Restricting to the sub-block corresponding to $V$, i.e. the one particle dynamics, the Hamiltonian increment is given by 
	\begin{equation}  
		dh_t:= dH_t\vert_V =
		\begin{pmatrix}
			0 & dW^{(1)}_t & 0 & \ldots & 0 & d\overline W^{(N)}_t \\ 
			d \overline W^{(1)}_t & 0 & dW^{(2)}_t & 0 & \ldots & 0 \\
			0 & d\overline W^{(2)}_t & 0 & dW^{(3)}_t &  & \vdots  \\ 
			\vdots 	&  &\ddots & \ddots & \ddots & 0 \\ 
			0& & &\ddots& 0 & d W_t^{(N-1)}\\
			dW^{(N)}_t & 0 & \ldots & 0 & d\overline W^{(N-1)}_t & 0
		\end{pmatrix}.
		\label{eq:oneparticleSDE}
	\end{equation}
	\label{ex:oneparticle}
\end{remark}
	
	\begin{remark}
		If we replace the fermionic Fock space and creation/annihilation operators by bosonic analogues, QSSEP becomes the quantum symmetric simple inclusion process (QSSIP). As the name suggests, when averaging over the noise, it reduces to the symmetric simple inclusion process \cite{bernard2025large}. It is also possible to consider Quantum Exclusion process with hopping
		amplitudes depending on the occupations of neighbouring sites \cite{bernard2026interacting}. 
	\end{remark}	

\begin{remark} 
	Given the amount of interest in the asymmetric simple exclusion process (ASEP) and in particular in its  totally asymmetric version (TASEP), it would be natural to define a quantum analogue. It may be checked that ${\overline {\mathcal L}}_\mathrm{ASEP}  = p \overrightarrow{\mathcal L} +q \overleftarrow{\mathcal L}$ acts as the ASEP Markov matrix on diagonal density matrices, which implies that one needs to consider a family of independent \emph{non-commutative} Brownian motions, i.e. Brownian motions such that It\=o contractions are computed using quadratic variations  $ \left\langle W_t^{(j)}, \overline W_t^{(j)} \right\rangle_t= p\, t$ and  $ \left\langle \overline W_t^{(j)}, W_t^{(j)} \right\rangle_t= q\, t$, see \cite{jin2020stochastic, robertson2021exact}. 
\end{remark}

\begin{remark} 
	One may wonder whether the QSSEP admits a natural discrete analogue. Of course, one may consider discrete dynamics  $\rho_{t+1} =e^{-\I (H_{t+1}-H_t)} \rho_t e^{\I (H_{t+1}-H_t) }$ where $H_t$ is the same as before,  or replacing Brownian motions by random walks. However, the averaged dynamics  $\overline{\rho_t}$ would no longer be described as  a discrete particle system with local moves.
Alternatively, one can consider random unitary circuits (locally conserving the number of particles) whose average dynamics are expected to be described by the macroscopic fluctuation theory of SSEP at large time \cite{Gullans2019,McCulloch2023}.
More generally, random quantum circuits are examples of discrete time stochastic quantum dynamics -- see  the survey \cite{Fisher2023}.
\end{remark}

\subsection{Open boundary conditions}
Let us now consider the system  in contact with reservoirs at its boundary,  which we call open boundary conditions. In the classical case, particles may enter or exit the system at sites $1$ and $N$ at certain exponential rates, as depicted below:  
\begin{center}
	\begin{tikzpicture}[scale=0.8]
		\draw[thick] (1,0) -- (10,0);
		\foreach \x in {1,...,10}{
		\draw (\x,-0.1) -- (\x,0.1);
		}
		\draw (1,-0.1) node[below]{\footnotesize $1$};
		\draw (2,-0.1) node[below]{\footnotesize $2$};
		\draw (3,-0.1) node[below]{\footnotesize $3$};
		\draw (10,-0.1) node[below]{\footnotesize $N$};
		\draw[thick, -stealth'] (0,0.5) to[bend left] (1,0.5);
		\draw (0.5,1) node {$\alpha_1$};
		\draw[thick,stealth'-] (0,-0.6) to[bend right] (1,-0.6);
		\draw (0.5,-1.1) node {$\beta_1$};
		\draw[thick,-stealth'] (10,0.5) to[bend left] (11,0.5);
		\draw (10.5,1) node {$\beta_N$};
		\draw[thick,stealth'-] (10,-0.6) to[bend right] (11,-0.6);
		\draw (10.5,-1.1) node {$\alpha_N$};
	\end{tikzpicture}
\end{center}
The dynamics of the open quantum SSEP is defined in a similar fashion as in the closed case,  except at the boundary where the gain/loss processes are defined using other CP maps \cite{bernard2019open}. The dynamics are defined by 
$$ d\rho_t = -\I [ dH_t, \rho_t] -\frac 1 2 [dH_t , [dH_t, \rho_t]]  + \mathcal L_{\rm bdry} \rho_t \,\mathrm d t ,$$ 
where now 
$$ dH_t = \sum_{j=1}^{N-1} c_{j+1}^{\dagger} c_j dW_t^{(j)}+ c_j^{\dagger} c_{j+1} d\overline W_t^{(j)} ,$$
with again one complex Brownian motion per edge, and $\mathcal L_{\rm bdry}$ is a boundary Lindbladian term 
$$ \mathcal L_{\rm bdry} = \alpha_1 \mathcal L_1^+ +\beta_1\mathcal L_1^-+\alpha_N \mathcal L_N^+ +\beta_N\mathcal L_N^-  ,$$
where (for $i=1,N$ the boundary sites)
 $$ \mathcal L_i^+ \rho = c_i^\dagger \rho c_i -\frac 1 2 \left( c_i c_i^\dagger \rho + \rho c_i c_i^\dagger \right),  $$ 
 and $\mathcal L_i^-$ is defined similarly exchanging $c_i$ and $c_i^\dagger$. 
 The operator $ \mathcal L_i^+$ (resp. $ \mathcal L_i^-$) injects (resp. extracts) a particle at the boundary site $i$ at unit rate, since $ \mathcal L_i^+ \ket{0}\bra{0}=\ket{1}\bra{1}-\ket{0}\bra{0}$  and $ \mathcal L_i^+ \ket{1}\bra{1}=0$ (and similarly for $ \mathcal L_i^-$).

\begin{remark}
	The restriction of the Hamiltonian increment to the one-particle sector is slightly different in the open case, compared to the closed case, to take the boundary conditions into account:
	\begin{equation} 
	dh_t:= dH_t\vert_V =
	\begin{pmatrix}
		0 & dW^{(1)}_t & 0 & \ldots & 0 &  0\\ 
		d \overline W^{(1)}_t & 0 & dW^{(2)}_t & 0 & \ldots & 0 \\
		0 & d\overline W^{(2)}_t & 0 & dW^{(3)}_t &  & \vdots  \\ 
	\vdots 	&  &\ddots & \ddots & \ddots & 0 \\ 
		0& & &\ddots& 0 & d W_t^{(N-1)}\\
		0 & 0 & \ldots & 0 & d\overline W^{(N-1)}_t & 0
	\end{pmatrix}.
\label{eq:oneparticleSDE-open}
\end{equation}
	\label{ex:oneparticle-open}
\end{remark}
 
 \subsection{Another approach via Wick's theorem} 
 \label{sec:Wick}
Let us assume that the initial state has the form 
\begin{equation} \label{eq:rho-quadratic}
 \rho_0 = \frac{1}{Z}e^{-\sum_{i,j=1}^N c_i^\dagger M_{i,j} c_j}, \;\;\;Z= \Tr\left( e^{-\sum_{i,j=1}^N c_i^\dagger M_{i,j} c_j} \right). 
 \end{equation}
for some Hermitian matrix $M$, possibly random. This is a reasonable assumption because this class is preserved by the dynamics. Further,  we believe that the stationary measure of the system is unique and does belong to that class -- see Section \ref{sec:openstationarymeasure}. 

For any $t$, the quantum probability that there are particles at locations $j_1, \dots, j_k$ is  $\Tr\left( \rho_t \hat n_{j_1} \dots \hat n_{j_k}\right)$, with number operators $\hat n_j= c_j^\dagger c_j$ (note that they are projectors, $\hat n_j^2=\hat n_j$).   By Wick's theorem   (for distinct indices $j_k$)
\begin{equation}
	\Tr\left( \rho_t \hat n_{j_1} \dots \hat n_{j_k}\right) = \det\left( \Tr(\rho_t \, c_{j_n}^\dagger c_{j_m} ) \right)_{n,m=1}^k.
	 \label{eq:applicationWick}
\end{equation}
This means that $\PP$-almost surely, the quantum SSEP configuration at any time is a determinantal point process (DPP) on $\lbrace 1, \dots, N\rbrace$ with a random correlation kernel 
$$ G_t(i,j) := \Tr( \rho_t\,  c_i^\dagger c_j ). $$ 
It is remarkable that the classical SSEP particle locations do not form a determinantal process but they may be realized as an explicit average of random determinantal processes, given by particles locations in QSSEP. 

For more general observables of the form $O= e^{ \sum_{i,j=1}^N c_i^\dagger A_{i,j}  c_j}$, we have
\begin{equation} \label{eq:rho-on-O}
 \Tr\left( \rho_t O  \right) = \det[I + G_t(e^A-I)],
 \end{equation}
so that $\rho_t$ is determined (at least its action on observables of the above form) by the random $N\times N$ matrix $G$. We note that \eqref{eq:rho-on-O} implies \eqref{eq:applicationWick}, by considering an expansion for small $A$. We thus have reduced the problem of understanding the $2^N\times 2^N$ density  matrix $\rho_t$ to the control of the $N\times N$ two-point correlation matrix $G_t$.

For closed or periodic boundary conditions, the evolution of $G_t$ is unitary,
\begin{equation}
	G_{t+\mathrm dt} = e^{-\I \mathrm dh_t} G_te^{\I \mathrm dh_t} ,
	\label{eq:evolutionG}
\end{equation}
 where $dh_t$ is the reduction of the Hamiltonian increment $dH_t$ to the one-particle sector,
\begin{equation}
	 \mathrm dh_t = \sum_{j=1}^N E_{j, j+1} \, \mathrm dW_t^{(j)}+ E_{j+1,j} \,\mathrm d\overline W_t^{(j)},
	 \label{eq:evolutiondh}
\end{equation}
 where $E_{i,j}$ is the $N\times N$ matrix with zero entries except a $1$ in position $(i,j)$  modulo $N$ (for the closed chain, the sum stops at $i=N-1$).   
  In other terms, $\mathrm dh_t$ is given by \eqref{eq:oneparticleSDE}  (periodic boundary conditions) or \eqref{eq:oneparticleSDE-open} (closed chain without reservoirs). Again, by It\={o}'s formula,
 \begin{equation}
 	dG_t =  -\I [dh_t, G_t] -\frac 1 2 \underbrace{[ dh_t, [dh_t, G_t]]}_{\text{It\={o} contraction}}.
 	\label{eq:ItoforG}
 \end{equation} 
 
In the open boundary case, one needs to adapt the one-particle Hamiltonian increment so that $dh_t$ is then given by \eqref{eq:oneparticleSDE-open} and to add boundary terms. Each coefficient of $G_t$ obeys a stochastic differential equation of the form (see Supplementary Materials in the arXiv version of \cite{bernard2019open})
\begin{subequations}
	\begin{align}
		\mathrm d G_t(i,i) &= \text{unitary dynamics } + \sum_{p\in\lbrace 1,N\rbrace} \delta_{pi}(\alpha_p(1-G_t(i,i))-\beta_p G_t(i,i) )\mathrm dt ,\\ 
		\mathrm d G_t(i,j) &= \text{unitary dynamics } -\frac 1 2  \sum_{p\in\lbrace 1,N\rbrace} (\alpha_p+\beta_p)(\delta_{pi}+\delta_{pj} )G_t(i,j)  \;\mathrm dt. 		
	\end{align}
	\label{eq:boundaryconditionforG}
\end{subequations}

 \begin{example}
	Take $N=2$ and $\alpha_1=\beta_1=\alpha_2=\beta_2=0$ (closed chain). There is only one edge and hence only one complex Brownian motion $W_t$. Then, $G_t$ is a $2\times 2$ Hermitian matrix (if $G_0$ is Hermitian) solving the system of SDE,
\begin{align*}
	\mathrm dG_t(1,1) &= -\I G_t(2,1) \mathrm dW_t + \I G_t(1,2) \mathrm d\overline W_t +(G_t(2,2)-G_t(1,1))\mathrm dt ,\\
	\mathrm d G_t(1,2) &= -\I (G_t(2,2)-G_t(1,1))\mathrm dW_t -G_{t}(1,2)\mathrm dt,
	\end{align*} 
	under the constraints $G_t(2,1)=\overline{G_t(1,2)}$ and $G_t(2,2)=n_1+n_2-G_t(1,1)$. The term $n_1+n_2$, equal to the total number of particles,   comes from the fact that $\Tr(G_t)=\Tr((c_1^\dagger c_1+c_2^\dagger c_2)\rho_t )=n_1+n_2$. Let us assume that the total number of particles is $1$ (it is conserved because the boundary parameters are all zero). Letting $D_t= G_t(1,1) - G_t(2,2)$, $R_t = \Re \left[ G_t(1,2)\right]$,  $I_t =  \Im \left[ G_t(1,2)\right]$ and let $W_t=\frac{B_1(t)+\I B_2(t)}{\sqrt 2}$, we arrive at 
\begin{equation}\begin{cases}
\mathrm dD_t &= -2D_t \,\mathrm dt + 2\sqrt 2\left( R_t \mathrm dB_2-I_t \mathrm dB_1\right),\\
\mathrm dR_t &= \frac{-1}{\sqrt 2} D_t \mathrm dB_2 - R_t\,\mathrm dt,\\
\mathrm dI_t &= \frac{1}{\sqrt 2} D_t \mathrm dB_1 - I_t \,\mathrm dt. 
	\end{cases}
	\label{eq:systemSDE}
\end{equation}
One may check that $D_t^2+4(R_t^2+I_t^2)$ is conserved, which of course we already knew  since it corresponds to the conservation of the determinant under the unitary dynamics.
\label{ex:N=2}
\end{example}

\begin{remark}
 	Here is another perspective from the point of view of representation theory.  The Lie algebra $\mathfrak{gl}_N(\mathbb C)$ acts on $\bigwedge V$ by mapping $E_{i,j}$ to $c_i^\dag c_j$. By exponentiation, this yields a representation of $GL_N(\mathbb C)$ on $\bigwedge V$. Let us denote by $\Gamma :GL_N(\mathbb C) \to \mathrm{End}( \bigwedge V)$ this representation map and  denote by $\gamma$ the Lie algebra representation map from $\mathfrak{gl}_N(\mathbb C)$ to $\mathrm{End}( \bigwedge V)$. Elements in $\Gamma(GL_N(\mathbb C))$ are of the form 
	$$\Gamma\left(e^A\right)=e^{ \sum_{i,j=1}^N c_i^\dagger A_{i,j}  c_j},$$ 
	as considered above, and the character formula is $\Tr(\Gamma(e^A))=\det(I+e^A)$.	
	We thus let the $*$-algebra $\mathcal A$ be generated by those elements and we restrict to density matrices $\rho=\frac{\Gamma(R)}{\Tr (\Gamma(R))}$ with $R=R^* \in GL_N(\mathbb C)$, as in \eqref{eq:rho-quadratic} above. Using the group law $\Gamma(R)\Gamma(O)=\Gamma(RO)$, we have, 
$$ \Tr\left( \rho\Gamma(O)\right) = \frac{\det(I+RO)}{\det(I+R)} = \det(I+ G(O-I)) , $$
as in \eqref{eq:rho-on-O}, with $G=\frac{R}{I+R}$. Taking $O=e^{\eps X}$ for $X\in  \mathfrak{gl}_N(\mathbb C)$ and letting $\eps$ going to zero  allows to identify $G$ as the $N\times N$ matrix 
$$ G_{i,j} =  \Tr \left(\rho \, \gamma(E_{i,j})\right)  = \Tr \left( \rho c_i^\dagger c_j\right),$$
which thus fully determines the density matrix.
\end{remark}

 We henceforth use the symbol `$\tr$' to denote the trace over the $N$-dimensional space (i.e. the one-particle sector), as opposed to the trace `$\Tr$' over the $2^N$-dimensional Fock space.

\subsection{Stationary measure of periodic/closed QSSEP and random matrix theory}
Clearly, the evolution defined by \eqref{eq:evolutionG} and \eqref{eq:evolutiondh} is unitary, so that the spectrum of $G_t$ is constant.  
This implies that $G_t$ is supported on the orbit of $G_0$ under the adjoint action of the unitary group. 

\begin{theorem}[{\cite[Proposition 1]{bauer2018equilibrium}}] Let $G_0$ be a $N\times N$ Hermitian matrix. The Haar distribution on the orbit of $G_0$ is the unique stationary measure of the stochastic differential equation on Hermitian matrices given by \eqref{eq:ItoforG} where $\mathrm dh_t$ is given by  \eqref{eq:evolutiondh}, supported on the unitary orbit of $G_0$. The same result holds when $\mathrm dh_t$ is given by \eqref{eq:oneparticleSDE-open}.
	\label{th:HaarOrbit}
\end{theorem}
\begin{proof}[Sketch of proof] 
First of all, the dynamics preserves the unitary orbit of $G_0$. This is clear from \eqref{eq:evolutionG} and it can be checked formally using the definition of the dynamics \eqref{eq:ItoforG}, by verifying  that $\tr(G_t^k)$ remains constant for all $k$, so that the spectrum is preserved.

 Thus, we are considering a Markovian diffusion on the unitary orbit of $G_0$. Since this space is compact, there exists at least one stationary probability measure. 
	
Then, the essential input is the fact that the set of matrices $E_{j, j+1}$ and $E_{j+1, j}$ for $1\leq j\leq N-1$ that appear in \eqref{eq:evolutiondh} forms a system of simple root generators of the Lie algebra $\mathfrak{su}(N)$. This implies that any stationary measure is invariant by unitary conjugation, hence it is unique and Haar distributed. 

To prove this $U(N)$ invariance, assume that $G$ is distributed according to a stationary measure, and  let
		$$ Z(A) = \mathbb E\left[e^{\tr(AG)} \right],$$
		where the expectation is taken over the stationary measure. Stationarity implies that 
		$$ Z(A) =\mathbb{E}\left[ Z(e^{\I \mathrm dh_t} A e^{-\I\mathrm dh_t})\right] $$
where the expectation is now taken with respect to Brownian motions in $\mathrm dh_t$.	
Applying It\={o}'s lemma, this implies that 
	\begin{equation}
		\left( \sum_{j=1}^{N-1} D_j^+D_j^-+ D_j^-D_j^+ \right) Z(A) =0,
		\label{eq:increment-su}
	\end{equation}
	where $D_j^+=D(E_{j,j+1}), D_j^-=D(E_{j+1,j})$ and for $\I X\in \mathfrak{su}(N)$, $D(X)$ is the operator 
	$$ D(X) F(A)  = \partial_sF(e^{\I s X}A e^{-\I sX}) \Big\vert_{s=0}.$$
	From \eqref{eq:increment-su}, one shows that $D_j^{\pm} Z(A)=0$ for all $1\leq j\leq N-1$. Because the $E_{j,j+1}$ and $E_{j+1,j}$ form a system of simple roots, one then concludes that $D(X) Z(A)=0$ for all $\I X\in \mathfrak{su}(N)$ which eventually implies that for any $U\in SU(N)$, 
	$$ Z(A) = Z(U^* A U).$$ 
	Details of the argument can be found in \cite[Appendix B]{bauer2018equilibrium}. 
\end{proof}

We stress that the Markov process describing the evolution of $G_t$ does not really have a unique stationary measure, in the sense that it depends on the initial condition $G_0$. As a consequence of Theorem \ref{th:HaarOrbit}, we expect that as $t$ goes to infinity, the law of $G_t$ weakly converges to the law of $U^* G_0 U$ where $U$ is a random Haar distributed unitary  matrix. We may then compute the Laplace transform of the probability measure of $G_t$ in the $t\to\infty$ limit as the Harish-Chandra--Itzykson--Zuber integral
\begin{equation} 
\lim_{t\to\infty}\mathbb E\left[ e^{N \tr(AG_t)} \right] = \int_{U(N)} \mathrm dU\,  e^{N \tr\left( A U^* G_0 U\right)}  = Z(NA) 	  ,
\end{equation}
where $\mathrm dU$ is Haar measure. A first contact with free probability emerges in the large $N$ limit: 
if $G_0$ admits a limiting spectral measure $\mu_0$, and $A$ has finite rank (and further assumptions), this integral over the unitary group becomes in the $N\to\infty$ limit the formal series \cite{itzykson1980planar, collins2002moments}
$$  \lim_{N\to\infty} \frac{1}{N} \log Z(zNA) =  \sum_{n=1}^{\infty} \frac{z^n}{n}  \tr(A^n) \kappa_n(\mu_0), $$
where $\kappa_n$ denote the free cumulants of $\mu_0$ which are defined in Section \ref{sec:cumulants}.

\begin{example} We reconsider the special case of Example \ref{ex:N=2}. 
	Let us fix the initial condition, for instance  
	$$G_0=\begin{pmatrix}
		1&0\\0&0
	\end{pmatrix}.$$ 
	For a Haar distributed unitary matrix $U$, it is well-known that $\vert U_{11}\vert^2$ is uniformly distributed in $[0,1]$ for $N=2$. Hence, by Theorem \ref{th:HaarOrbit}, $G(1,1)$ should be uniformly distributed in $[0,1]$ for large $t$. Equivalently, $D_t$ in Example \ref{ex:N=2} should be uniform in $[-1,1]$. To check that, it is not clear that the system of SDE \eqref{eq:systemSDE} can be solved explicitly, but we do find on numerical simulations (see Figure \ref{fig:simulations}) that the probability distribution of $D_t$ approaches the uniform one. 
	\begin{figure}
		\includegraphics[width=0.3\textwidth]{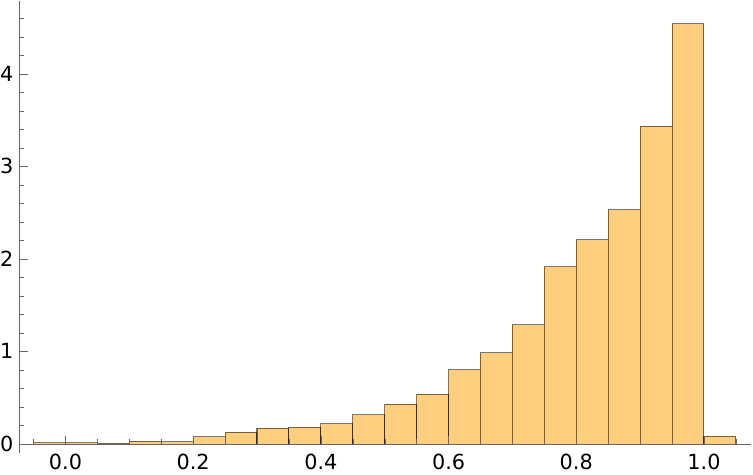}
				\includegraphics[width=0.3\textwidth]{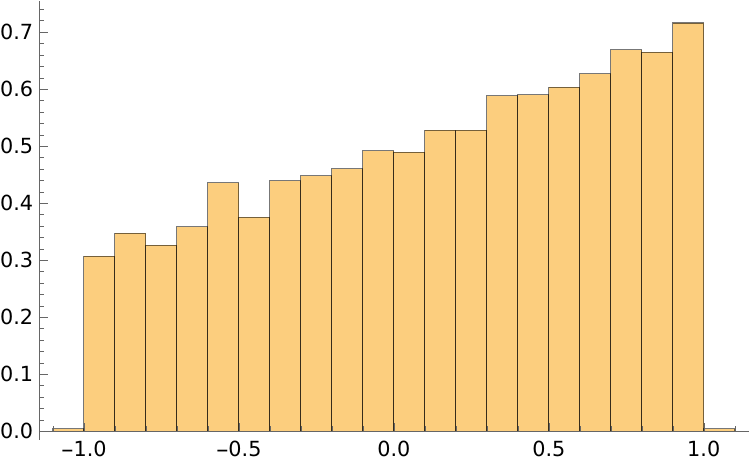}
						\includegraphics[width=0.3\textwidth]{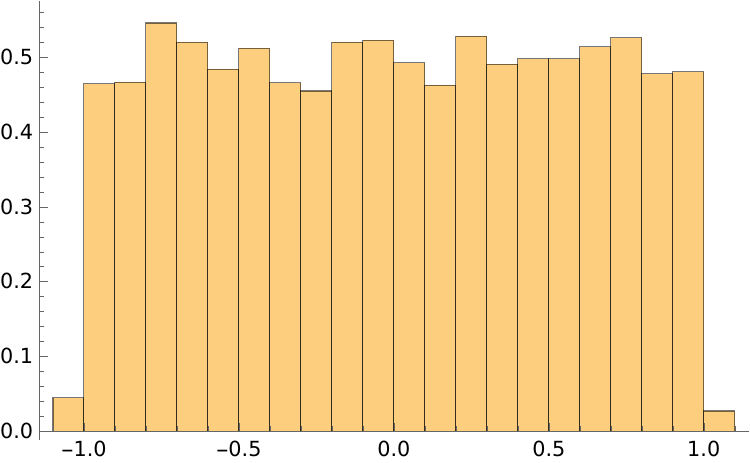}
		\caption{Probability distribution of $D_t$ for $t=0.1, t=1$ and $t=10$. Initially the distribution is concentrated at $1$, and it converges to the uniform distribution on $[-1,1]$. The few values outside $[-1,1]$ are numerical errors due to the simulated dynamics being not perfectly unitary.}
		\label{fig:simulations}
	\end{figure}
\end{example}

\begin{remark}
The solution to the SDE $U_{t+\mathrm dt} = e^{-\I\mathrm dh_t}U_t$ is different from the standard notion of Brownian motion on $U(N)$  or $SU(N)$. This stochastic process, as well as the process $G_{t+\mathrm dt} = e^{-\I \mathrm dh_t} G_te^{\I \mathrm dh_t}$ from \eqref{eq:evolutionG}, does not seem to have been considered in the mathematics literature, and it seems  interesting to study it further. In particular, one could compare the mixing time of these diffusions to the mixing time of the Brownian motion on $U(N)$ \cite{saloff2004convergence, meliot2012cut}. In the context of random walks on the symmetric group, the analogous problem  is  well-known (and solved): the mixing time is different when the walk is generated by all transpositions \cite{diaconis1981generating} or only by adjacent transpositions \cite{lacoin2016mixing} --- the latter case being related to the classical SSEP (with closed boundary conditions).
\label{rem:mixing}
\end{remark}

\section{Open QSSEP and free probability}
\label{sec:openQSSEP}
\subsection{Preliminaries about classical and free cumulants}
\label{sec:cumulants}
For random variables $X_1, \dots, X_n$, and any $(k_1, \dots, k_n) \in \Z_{\geq 0}^n$, the classical cumulants $\Kclassical_{k_1, \dots, k_n}(X_1, \dots, X_n)$ are generally defined through the coefficients in the power series 
$$ \log \mathbf E\left[ e^{t_1X_1+\dots+t_nX_n}\right] = \sum_{k_1, \dots, k_n \geq 0} \Kclassical_{k_1, \dots, k_n} (X_1, \dots, X_n) \prod_{i=1}^n \frac{t_i^{k_i}}{k_i!} . $$
The cumulants $\Kclassical(X_1, \dots, X_n):=\Kclassical_{1,\dots,1}(X_1, \dots, X_n) $ are related to the moments via 
\begin{equation}	
	\mathbf E[X_1 \cdots X_n] = \sum_{\pi \in \mathcal P_n}  \prod_{B\in \pi} \Kclassical\left(\lbrace X_i \rbrace_{i\in B}\right) .
	\label{eq:classicalcumulantstomoments}
\end{equation}
where the sum runs over $\mathcal P_n$, the set of partitions $\pi$ of  $\lbrace 1, \dots, n\rbrace$ into disjoint subsets denoted $B$. Equivalently, cumulants are expressed in terms of moments  as 
\begin{equation}	
	\Kclassical(X_1, \dots, X_n) = \sum_{\pi\in \mathcal P_n} \left(\vert\pi\vert -1\right)! (-1)^{\vert \pi\vert -1} \prod_{B\in \pi} \mathbf E\left[ \prod_{i\in B} X_i\right],
	\label{eq:classicalmomentstocumulants}
\end{equation}
where $\vert\pi\vert $ denotes the number of blocks in the partition. In particular, the first cumulant is the expectation $\Kclassical(X) = \mathbf E[X]$, the second cumulant is the covariance 
\begin{equation}
	\Kclassical(X,Y)= \mathbf E[XY] - \mathbf E[X] \mathbf E[Y],
	\label{eq:secondclassicalcumulant}
\end{equation}
and the third cumulant is given by 
\begin{equation}
	\Kclassical(X,Y,Z)= \mathbf E[XYZ] - \mathbf E[X] \mathbf E[YZ] - \mathbf E[Y] \mathbf E[ZX] - \mathbf E[Z] \mathbf E[XY] + 2 \mathbf E[X]\mathbf E[Y]\mathbf E[Z].
	\label{eq:thirdclassicalcumulant}
\end{equation}
A well known property of classical cumulants is that when  the variables $X_i$ are independent,  $\Kclassical_k(X_1+\dots + X_n) = \Kclassical_k(X_1)+ \dots + \Kclassical_k(X_n)$ for all $k\geq 1$.

There exist analogous properties in free probability theory. In this context,  random variables become elements of a unital associative $*$-algebra $\mathcal A$ over $\mathbb C$ equipped with a linear map  (non-commutative expectation) $\varphi: \mathcal A\to \mathbb C$ such that $\varphi(1)=1$ and $\varphi(a^*a)\geq 0$ for $a\in\mathcal{A}$. The free cumulants $\kappa (a_1, \dots, a_n)$ for $a_i\in \mathcal A$ are then defined implicitly via 
\begin{equation}	
	\varphi(a_1 \cdots a_n) = \sum_{\pi \in \mathrm{NC}(\mathcal P_n)} \prod_{B\in \pi} \kappa\left( (a_i)_{i\in B}\right),
	\label{eq:deffreecumulants}
\end{equation}
 where $\mathrm{NC}(\mathcal P_n)$ is the subset of $\mathcal P_n$ consisting of non-crossing partitions (the notion of non-crossing partition can be understood from the examples in Figure \ref{fig:noncrossing}).
 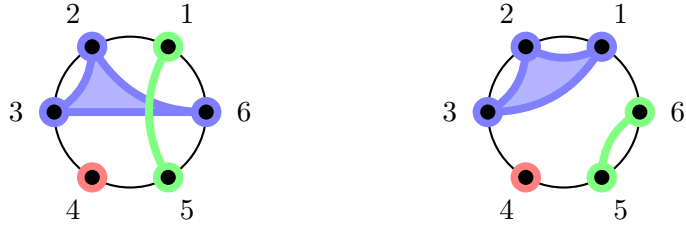
\begin{figure}
 	\begin{center} 
 		\begin{tikzpicture}
 			\draw[thick] (1,0) arc(0:360:1);
 			\filldraw[line width=3pt, draw= blue!50, fill=blue!30] (0:1) to[bend left] (2*60:1) to[bend left] (3*60:1) -- cycle;
 			\fill[blue!50] (0:1) circle(0.2);
 			\fill[blue!50] (2*60:1) circle(0.2);
 			\fill[blue!50] (3*60:1) circle(0.2);
 			\draw[line width=3pt, green!50] (60:1) to[bend right] (-60:1);
 			\fill[green!50] (60:1) circle(0.2);
 			\fill[green!50] (-60:1) circle(0.2);
 			\fill[red!50] (-120:1) circle(0.2);
 			\foreach \x in {1,...,6}
 			{\fill (\x*60:1) circle(0.1);
 				\draw (\x*60:1.5) node{$\x$};}
 		\end{tikzpicture}
 		\hspace{2cm}
 		 \begin{tikzpicture}
 			\draw[thick] (1,0) arc(0:360:1);
 			\filldraw[line width=3pt, draw= blue!50, fill=blue!30] (60:1) to[bend left] (2*60:1) to[bend left] (3*60:1) to[bend right] (60:1);
 			\fill[blue!50] (60:1) circle(0.2);
 			\fill[blue!50] (2*60:1) circle(0.2);
 			\fill[blue!50] (3*60:1) circle(0.2);
 			\draw[line width=3pt, green!50] (0:1) to[bend right] (-60:1);
 			\fill[green!50] (0:1) circle(0.2);
 			\fill[green!50] (-60:1) circle(0.2);
 			\fill[red!50] (-120:1) circle(0.2);
 			\foreach \x in {1,...,6}
 			{\fill (\x*60:1) circle(0.1);
 				\draw (\x*60:1.5) node{$\x$};}
 		\end{tikzpicture}
 	\end{center} 
 	\caption{Left: A partition of $\lbrace 1,2,\dots, 6\rbrace$ into 3 blocks (green, blue and red). We see that the segment joining the green points  and segments joining blue points do cross. Right: A non-crossing partition.} 
 	\label{fig:noncrossing}
 \end{figure}
 As in \eqref{eq:classicalmomentstocumulants}, the relation \eqref{eq:deffreecumulants} can be inverted to express free cumulants in terms of free moments. One has $\kappa(a) = \varphi(a)$ and 
 \begin{align*}
 	\kappa(a,b) &= \varphi(ab)-\varphi(a)\varphi(b),\\ 
 	\kappa(a,b,c)&= \varphi(abc)- \varphi(a)\varphi(bc) - \varphi(b)\varphi(ca)  - \varphi(c)\varphi(ab)  + 2\varphi(a)\varphi(b)\varphi(c),
 \end{align*}
 in a way similar to the above classical moment-cumulant formula.
 However, at order 4,  the free cumulant $\kappa(a_1, a_2, a_3, a_4)$ no longer has the same expression as the classical cumulant $\mathbf K(X_1, X_2, X_3, X_4)$,  because there exists a crossing partition of order $4$: 
 \begin{center} 
 	\begin{tikzpicture}
 	\draw (1,0) arc(0:360:1);
 	\draw[line width=3pt, blue!50] (45:1) -- (45+2*90:1);
 	\fill[blue!50] (45:1) circle(0.2);
 	\fill[blue!50] (45+2*90:1) circle(0.2);
 	\draw[line width=3pt, green!50] (45+90:1) -- (45+3*90:1);
 	\fill[green!50] (45+90:1) circle(0.2);
 	\fill[green!50] (45+3*90:1) circle(0.2);
 	\fill (45:1) circle(0.1);
 	\fill (135:1) circle(0.1);
 	\fill ({45+2*90}:1) circle(0.1);
  	\fill ({45+3*90}:1) circle(0.1);
  	\draw (45:1.5) node{$1$};
  	\draw (135:1.5) node{$2$};
  	\draw ({45+2*90}:1.5) node{$3$};
  	\draw ({45+3*90}:1.5) node{$4$};
 \end{tikzpicture}
 \end{center} 
 Adapting the notion of independent variables to the non-commutative setting, elements $a_1,\, a_2\in\mathcal{A}$ are said to be free (or freely independent) iff all their mixed free cumulants vanish.
 We refer to \cite{Nica2006Lecture,speicher2019lecture} for more details on free probability.

\subsection{First moments of the stationary measure}
Let us go back to quantum exclusion processes and consider the case of open QSSEP with boundary parameters chosen so as to impose densities $n_a< n_b$ at both ends (the relation is  $n_a=\frac{\alpha_1}{\alpha_1+\beta_1}$ and $n_b=\frac{\alpha_N}{\alpha_N+\beta_N}$ as for the SSEP). We observe that in \eqref{eq:ItoforG} and \eqref{eq:boundaryconditionforG}, all terms are linear in the entries of the matrix $G$, so that one can compute the expectation under the stationary measure  $\mathbb E[G(i,j)]$ by imposing that the expectation of all such coefficients remain constant over time, leading to  solving a system of linear equations. In this section, all expectations are with respect to the stationary measure of the process $G_t$. One obtains  that $\mathbb E[G(i,j)]=0$ for $i\neq j$ and for $i=\lfloor xN\rfloor$  \cite{bernard2019open} 
$$ \mathbb E \left[ G(i,i)\right] = n_a+ x(n_b-n_a)+ O\left( \frac{1}{N}\right).$$
At second order, again, one can compute all the terms of the form $\mathbb E\left[G(i,j)G(k,\ell) \right]$ through a linear system. For $i=\lfloor xN\rfloor$,  $j=\lfloor yN\rfloor$, the dominant terms have the form $\mathbb E\left[G(i,j) G(j,i)\right]$ and one finds \cite{bernard2019open}, letting $\Delta n:=n_b-n_a$,
\begin{equation}
	\mathbb K\left[G(i,j), G(j,i)\right]   = \frac{(\Delta n)^2}{N} x(1-y) + O\left( \frac{1}{N^2}\right), \text{ for }x<y, 
	\label{eq:covariance}
\end{equation} 
and
\begin{align*}
	\mathbb K\left[G(i,i), G(i,i)\right]   &=  \frac{(\Delta n)^2}{N} x(1-x) + O\left( \frac{1}{N^2}\right), \\
	\mathbb K\left[G(i,i), G(j,j)\right]   &= -\frac{(\Delta n)^2}{N^2} x(1-y) + O\left( \frac{1}{N^3}\right), \text{ for }x<y,
\end{align*} 
where  $\mathbb K(X_1(\omega), \dots, X_n(\omega))$ denotes the classical cumulants of random variables $X_1, \dots, X_n$ with respect to the measure and expectation $\mathbb P, \mathbb E$. For example, \eqref{eq:covariance} is the covariance between random variables $G(i,j)$ and $G(j,i)$.  It is useful to note that the polynomial $x(1-y)$ may also be written as 
\begin{equation}
	x\wedge y-xy
	\label{eq:secondmoment}
\end{equation}
to remove the constraint that $x<y$.  Here $x\wedge y= \mathrm{min}(x,y)$. At third order, dominant terms have the form 
$$ \mathbb K\left[G(i,j), G(j,k) , G(k,i) \right]  = \frac{(\Delta n)^3}{N^2} x(1-2y)(1-z)+ O\left( \frac{1}{N^3}\right),\text{ for }x<y<z, $$
with $i=\lfloor xN\rfloor$,  $j=\lfloor yN\rfloor$ and  $k=\lfloor zN\rfloor$. Again, this  can be rewritten in a more symmetric way as 
\begin{equation}
	x\wedge y\wedge z- x(y\wedge z) - y(z\wedge x) - z(x\wedge y) + 2 xyz.
	\label{eq:thirdmoment}
\end{equation}
The formulas \eqref{eq:secondmoment} and \eqref{eq:thirdmoment} are reminiscent of expressions \eqref{eq:secondclassicalcumulant} and \eqref{eq:thirdclassicalcumulant} for classical cumulants in terms of moments. But as we will see, things become slightly different at the fourth order and higher where free probability comes into play. 

\subsection{General case and free probability}
Let us turn to the general case.  It is convenient to associate to a product of the form  $G(i_1,j_1)G(i_2,j_2) \dots G(i_r,j_r)$ an oriented graph $\mathcal G$ with edges $(i_k, j_k)$.  It turns out that very few graphs have a nontrivial average. We have 
$\mathbb E\left[ \mathcal G\right] =0 $
unless $\mathcal G$ is an Eulerian graph (at each vertex, the number of ingoing and outgoing arrows are the same). This follows from $U(1)^N$ invariance, i.e. we can change $c_j$ into $c_je^{\I\theta_j}$ so that $c_j^\dagger$ is changed into $c_j^\dagger e^{-\I\theta_j}$ without changing the probability distribution of $G(i,j)$. Taking the  $\theta_j$ as  independent  uniformly random phases, one can see that the expectation vanishes if the graph is not Eulerian. Again, all expectations in this section are with respect to the stationary measure for the process $G_t$.

We have already seen that 
$$ \mathbb E \left[\raisebox{-3.5ex}{\begin{tikzpicture}[scale=0.3] 
\draw[thick] (0,0) arc (-90:270:1);
\fill (0,0) circle(0.2);
\node[below] at (0,0) {$i$};
\end{tikzpicture}} \right]  =n_a+ x(n_b-n_a) + O\left( \frac{1}{N} \right).$$
At second order, we have seen in \eqref{eq:covariance} that (recall $\Delta n=n_b-n_a$), 
$$ \mathbb K \left[\raisebox{-1.5ex}{\begin{tikzpicture}[yscale=0.35, xscale=0.35] 
		\draw[thick] (1,0) arc (0:360:1);
		\fill (0:1) circle(0.2);
		\fill (180:1) circle(0.2);
		\node[left] at (180:1) {$i$};
		\node[right] at (0:1) {$j$};
\end{tikzpicture}} \right]  =\frac{(\Delta n)^2}{N} x(1-y)+ O\left( \frac{1}{N^2} \right), \text{ for }x<y,$$
and we recall that $\mathbb K$ denotes the cumulants of the variables represented by the graph, i.e. 
$$ \mathbb K \left[\raisebox{-1.5ex}{\begin{tikzpicture}[yscale=0.35, xscale=0.35] 
		\draw[thick] (1,0) arc (0:360:1);
		\fill (0:1) circle(0.2);
		\fill (180:1) circle(0.2);
		\node[left] at (180:1) {$i$};
		\node[right] at (0:1) {$j$};
\end{tikzpicture}} \right]  =\mathbb E  \left[\raisebox{-1.5ex}{\begin{tikzpicture}[yscale=0.35, xscale=0.35] 
\draw[thick] (1,0) arc (0:360:1);
\fill (0:1) circle(0.2);
\fill (180:1) circle(0.2);
\node[left] at (180:1) {$i$};
\node[right] at (0:1) {$j$};
\end{tikzpicture}} \right]  
-
\mathbb E \left[\raisebox{-1.5ex}{\begin{tikzpicture}[yscale=0.35, xscale=0.35] 
\draw[thick] (1,0) arc (0:180:1);
\fill (0:1) circle(0.2);
\fill (180:1) circle(0.2);
\node[left] at (180:1) {$i$};
\node[right] at (0:1) {$j$};
\end{tikzpicture}} \right]  
\mathbb E \left[\raisebox{-1.5ex}{\begin{tikzpicture}[yscale=0.35, xscale=0.35] 
\draw[thick] (-1,0) arc (180:360:1);
\fill (0:1) circle(0.2);
\fill (180:1) circle(0.2);
\node[left] at (180:1) {$i$};
\node[right] at (0:1) {$j$};
\end{tikzpicture}} \right]. $$
At third order, we have 
$$\mathbb K \left[\raisebox{-4ex}{\begin{tikzpicture}[yscale=0.35, xscale=0.35] 
		\draw[thick] (1,0) arc (0:360:1);
		\fill (30:1) circle(0.2);
		\fill (150:1) circle(0.2);
		\fill (270:1) circle(0.2);
		\node[below] at (270:1) {$z$};
		\node[above right] at (30:0.8) {$y$};
		\node[above left] at (150:0.8) {$x$};
\end{tikzpicture}} \right]  = \frac{(\Delta n)^3}{N^2} x(1-2y)(1-z) +O\left( \frac{1}{N^3} \right), \text{ for }x<y<z.$$
We claim that at any order $p$, the leading order terms are those corresponding to cyclic loop (and its pinchings). Let 
\begin{equation}
	g_p(x_1, \dots, x_p) :=\lim_{N\to\infty} N^{1-p} \, \mathbb K
	\left[ \raisebox{-7ex}{\begin{tikzpicture}[xscale=0.8,yscale=0.8]
			\draw[thick] (1,0) arc (0:360:1);
			\foreach \x in {1,2,...,10}
			{\fill (\x*36:1) circle(0.1);}
			\foreach \x in {1,2,...,4}
			\draw ({-\x*36+4*36}:1.4) node{$x_{\x}$};
			\draw ({6*36}:1.4) node[rotate=-50]{$\ldots$};
			\draw ({7*36}:1.4) node[rotate=-15]{$\ldots$};
			\draw ({8*36}:1.4) node[rotate=20]{$\ldots$};
			\draw ({9*36}:1.4) node[rotate=55]{$\ldots$};
			\draw ({5*36}:1.4) node[rotate=90]{$\ldots$};
			\draw (4*36:1.4) node{$x_p$};
	\end{tikzpicture}} \right], 
\label{eq:defg}
\end{equation} 
which we assume to be computed with all discrete indices $i_k$ distinct, although some of the $x_k$ may be equal in the limit. 

One can compute the cumulants of cyclic loops order by order \cite{bernard2019open,bernard2021solution}.
For example, when $p=4$ and $x_1<x_2<x_3<x_4$ we have
$$\mathbb K \left[\raisebox{-5ex}{\begin{tikzpicture}[yscale=0.5, xscale=0.5] 
		\draw[thick] (1.2,0) arc (0:360:1.2);
		\foreach \x in{0,1,2,3}
		{
		\fill (\x*90+45:1.2) circle(0.15);
		}
	\draw (135:1.7) node{$x_1$};
	\draw (135-90:1.7) node{$x_2$};
	\draw (135-180:1.7) node{$x_3$};
	\draw (135-260:1.7) node{$x_4$};
\end{tikzpicture}} \right]  = \frac{(\Delta n)^4}{N^3} x_1(1-3x_2-2x_3+5x_2x_3)(1-x_4) +O\left( \frac{1}{N^4} \right),$$
and 
$$\mathbb K \left[\raisebox{-5ex}{\begin{tikzpicture}[yscale=0.5, xscale=0.5] 
		\draw[thick] (1.2,0) arc (0:360:1.2);
		\foreach \x in{0,1,2,3}
		{
			\fill (\x*90+45:1.2) circle(0.15);
		}
		\draw (135:1.7) node{$x_1$};
		\draw (135-90:1.7) node{$x_3$};
		\draw (135-180:1.7) node{$x_2$};
		\draw (135-260:1.7) node{$x_4$};
\end{tikzpicture}} \right]  = \frac{(\Delta n)^4}{N^3} x_1(1-4x_2-x_3+5x_2x_3)(1-x_4) +O\left( \frac{1}{N^4} \right).$$
Note that, if the points $x_k$ are ordered on the line, these cumulants  then depend on the ordering on the loop. They are all polynomials in the $x_k$ with integer coefficients.

We now give a general expression for $g_p$ in terms of free cumulants. From now on we restrict to the case $n_a=0, n_b=1$, though the most general case is not much different. Let $[0,1]$ be equipped with the Lebesgue measure, denoted $\varphi$, and consider the variables $I_{x}:= \mathds{1}_{[0,x]}$. We have 
$$ \varphi(I_{x_1} \dots I_{x_p}) = x_1 \wedge \dots \wedge x_p.$$ 

\begin{claim} The function $g_p$ defined by \eqref{eq:defg} is a polynomial of degree $1$ in each variable. It is explicitly given by 
\begin{equation}
	g_p(x) = \kappa_p(I_{x_1}, \dots, I_{x_p}),
	\label{eq:formulagp}
\end{equation}
where $\kappa_p$ is the $p$-th free cumulant of the family of (commuting) variables $I_{x_1}, \dots, I_{x_p}$. 
\label{claim:freecumulantsg}
\end{claim} 
Claim \ref{claim:freecumulantsg} was first obtained in \cite{biane2021combinatorics}, based on  the combinatorial analysis of a recursive  characterization of the functions $g_p$, found in  \cite{bernard2021solution}. Fully proving this characterization remains an open mathematical problem. However, the  deduction of \eqref{eq:formulagp} from the recursive characterization is a mathematical theorem \cite[Theorem 6.1]{biane2021combinatorics}. We will explain in Section \ref{sec:elementsproof} why free cumulants arise. We will not follow the combinatorial analysis of \cite{biane2021combinatorics} but rather present a more direct derivation from \cite{hruza2022coherent} based on free probability. 

\begin{remark}
The scaling \eqref{eq:defg} of the $g_p$'s with $1/N$  suggests a large deviation behaviour 
$$ \lim_{N\to\infty} \frac{1}{N} \log \mathbb E\left[ e^{N\tr(AG)}\right] = \mathcal F(A),$$
where the functional $\mathcal F$ is related to the polynomials $g_p$ \cite{bernard2023structured}.
\end{remark}

\begin{figure}
	\begin{tikzpicture}[xscale=0.75,yscale=0.75]
		\draw[thick] (2,0) arc(0:360:2);
		\filldraw[line width=3pt, draw= blue!50, fill=blue!30] (36:2) to[bend left] (2*36:2) to[bend left=20] (4*36:2) to[bend right=10] (36:2);
		\fill[blue!50] (36:2) circle(0.2);
		\fill[blue!50] (2*36:2) circle(0.2);
		\fill[blue!50] (4*36:2) circle(0.2);
		\draw[line width=3pt, green!50] (-36:2) to[bend right] (-3*36:2);
		\fill[green!50] (-36:2) circle(0.2);
		\fill[green!50] (-3*36:2) circle(0.2);
		\draw[line width=3pt, red!50] (-4*36:2) to[bend left=10] (0:2);
		\fill[red!50] (-4*36:2) circle(0.2);
		\fill[red!50] (0:2) circle(0.2);
		\fill[orange!50] (-2*36:2) circle(0.2);
		\fill[magenta!50] (3*36:2) circle(0.2);
		\fill[cyan!50] (5*36:2) circle(0.2);
		\foreach \x in {1,...,10}
		{\fill (\x*36:2) circle(0.1);
			\draw (\x*36:2.5) node{$\bar \x$};}
		\draw[dashed, line width=1pt, gray!40] (2*36+18:2) to[bend left=40] (3*36+18:2);
		\draw[dashed, line width=1pt, gray!40] (4*36+18:2) to[bend left=40] (5*36+18:2);
		\draw[dashed, line width=1pt, gray!40] (0*36+18:2) to[bend left=10] (4*36+18:2);
		\draw[dashed, line width=1pt, gray!40] (0*36+18:2) to[bend right=10] (5*36+18:2);
		\draw[dashed, line width=1pt, gray!40] (-1*36+18:2) to[bend right=15] (-4*36+18:2);
		\draw[dashed, line width=1pt, gray!40] (-2*36+18:2) to[bend right=40] (-3*36+18:2);
	\end{tikzpicture}
	\hspace{2cm}
	\begin{tikzpicture}[xscale=0.75,yscale=0.75]
		\draw[thick] (2,0) arc(0:360:2);
		\begin{scope}[gray]
			\draw[line width=2pt, draw= gray!50, fill=gray!30] (36:2) to[bend left] (2*36:2) to[bend left=20] (4*36:2) to[bend right=10] (36:2);
			\draw[line width=2pt, gray!50] (-36:2) to[bend right] (-3*36:2);
			\draw[line width=2pt, gray!50] (-4*36:2) to[bend left=10] (0:2);
			\foreach \x in {1,...,10}
			{\fill[gray!50] (\x*36:2) circle(0.1);
				\draw[gray!50] (\x*36:2.5) node{$\bar \x$};}
		\end{scope}
		
		\foreach \x in {1,...,10}
		{\fill (\x*36+18:2) circle(0.1);
			\draw[] (\x*36+18:2.5) node{$\x$};}
		\draw[line width=2pt] (2*36+18:2) to[bend left=40] (3*36+18:2);
		\draw[line width=2pt] (4*36+18:2) to[bend left=40] (5*36+18:2);
		\draw[line width=2pt] (0*36+18:2) to[bend left=10] (4*36+18:2);
		\draw[line width=2pt] (0*36+18:2) to[bend right=10] (5*36+18:2);
		\draw[line width=2pt] (-1*36+18:2) to[bend right=15] (-4*36+18:2);
		\draw[line width=2pt] (-2*36+18:2) to[bend right=40] (-3*36+18:2);
	\end{tikzpicture}
	\caption{Left: a non-crossing partition in $\mathcal P_{10}$. Right: Its Kreweras complement is depicted in black. 
		In applications to cyclic loop expectations in QSSEP or random matrix theory, the partition $\pi$ is partitioning the edges and its dual the points joining the edges into a loop (this is why the elements on the left are denoted $\bar 1, \bar 2, \dots$ instead of $1,2,\dots$).
		} 
	\label{fig:Kreweras}
\end{figure}
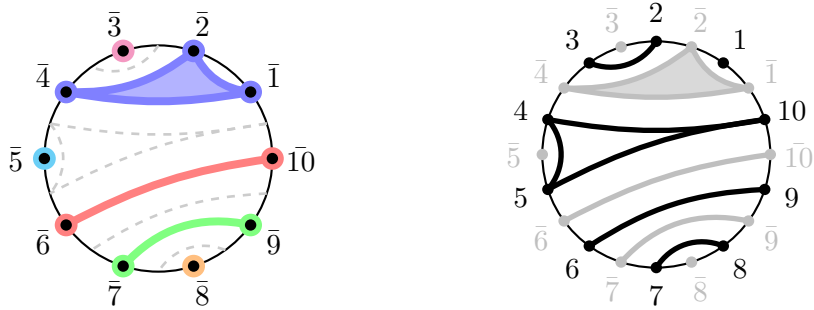

\subsection{Elements of proof}
\label{sec:elementsproof}

Let us explain now how free cumulants arise. First, we recall that the cyclic polynomials $g_p$ are defined in \eqref{eq:defg} as classical cumulants. Using \eqref{eq:classicalcumulantstomoments}, we have that 
$$ \mathbb E\left[G(i_1, i_2) \ldots G(i_p,i_1) \right]  = \sum_{\pi\in \mathcal P_p}  \prod_{B\in \pi} \mathbb K[(G(i_j, i_{j+1}))_{j\in B}].$$
We first argue that not all cumulants contribute in this sum, and that those which contribute are associated to non-crossing partitions. When $p=2$, 
$$\mathbb E  \left[\raisebox{-2.5ex}{\begin{tikzpicture}[yscale=0.5, xscale=0.5] 
		\draw[thick] (1,0) arc (0:360:1);
		\fill (0:1) circle(0.2);
		\fill (180:1) circle(0.2);
		\node[left] at (180:1) {$i$};
		\node[right] at (0:1) {$j$};
\end{tikzpicture}} \right]  = 
\underbrace{\mathbb K \left[G(i,j),G(j,i) \right]}_{N^{-1}g_2(x,y)}+ \underbrace{\mathbb K\left[G(i,j)\right] \mathbb K\left[G(j,i)\right] }_{\delta_{i,j} g_1(x)g_1(y)}. 
$$
At large scale, the indicator function $\delta_{ij}$ should be viewed as a discrete approximation of $N^{-1} \delta(x-y)$ so that as $N\to\infty$, 
$$\mathbb E  \left[\raisebox{-2.5ex}{\begin{tikzpicture}[yscale=0.5, xscale=0.5] 
		\draw[thick] (1,0) arc (0:360:1);
		\fill (0:1) circle(0.2);
		\fill (180:1) circle(0.2);
		\node[left] at (180:1) {$x$};
		\node[right] at (0:1) {$y$};
\end{tikzpicture}} \right] \approx N^{-1}\left(g_2(x,y) + \delta(x-y)g_1(x) g_1(y) \right),$$
to be understood in a weak sense. 
More generally let  us define the expectations (not the cumulants) of  cyclic loops
$$ g^s(x):=\lim_{N\to\infty} N^{p-1} \mathbb E\left[G(i_1, i_2) \ldots G(i_p,i_1)\right],$$
where $i_j= \lfloor Nx_j\rfloor$ as before and $x=(x_1,\cdots,x_p)$. 
	$U(1)^N$ invariance and scalings with $1/N$ imply the following
\begin{claim}[\cite{hruza2022coherent}] \label{claim:noncrossing}
	We have 
\begin{equation}
	g^s(x) = \sum_{\pi\in \mathrm{NC}(\mathcal P_p)} g_{\pi}(x) \delta_{\pi^*}(x),
	\label{eq:keyrelation}
\end{equation}
where the sum is over non-crossing partition $\pi$ and $g_{\pi}$ is the product of the cyclic cumulants associated to each parts of the partition $\pi$,
\begin{equation}
	g_{\pi}(x_1, \dots, x_p)= \prod_{B\in \pi } g_{\vert B\vert}(x_B), \;\;\;x_B=(x_j)_{j\in B},
	\label{eq:defgpi}
\end{equation}  
and the partition $\pi^\ast$ denotes the Kreweras complement of the non-crossing partition $\pi$ (see Figure \ref{fig:Kreweras}) and $\delta_{\pi^*}(x)= \prod_{B\in \pi^*} \delta_{\vert B\vert}(x_B)$ where $\delta_n(x_1, \dots, x_n) := \delta(x_1-x_2)\delta(x_2-x_3) \dots \delta(x_{n-1}-x_n)$. 
\end{claim}

It is useful to understand $\pi\in \mathrm{NC}(\mathcal P_p)$ as a partition of the edges and its dual $\pi^\ast$ as a partition of the points. For a loop the number of edges is equal to the number of points.
The identity \eqref{eq:keyrelation} should be understood in a weak sense, i.e. that for any continuous function $\psi :[0,1]^p \to\R$, 
$$ \int_{[0,1]^p}\mathrm dx_1 \dots\,\mathrm dx_p \; \psi(x) g^s(x) = \sum_{\pi\in \mathrm{NC}(\mathcal P_p)} \int_{[0,1]^p}\!\!\!\mathrm dx_1 \dots\,\mathrm dx_p \; \psi(x) g_{\pi}(x) \delta_{\pi^*}(x).$$
Alternatively, one can consider $\lim_{N\to\infty} N^{-1}\mathbb{E}[ \tr(G\Delta_1 \cdots G \Delta_p)]$, with $\Delta_1,\cdots,\Delta_p$ diagonal matrices, see Claim \ref{claim:traces}.

 Let us now explain the origin of non-crossing partitions and how \eqref {eq:keyrelation} is obtained for $p=4$. The fourth order cyclic moment can be expressed, by \eqref{eq:classicalcumulantstomoments}, as a sum over partitions of products of cumulants: 
\begin{multline*}
\mathbb E \left[\raisebox{-6ex}{\begin{tikzpicture}[xscale=0.7,yscale=0.7] 
		\draw[thick] (1.2,0) arc (0:360:1.2);
		\foreach \x in{0,1,2,3}
		{
			\fill (\x*90+45:1.2) circle(0.1);
		}
		\draw (135:1.6) node{$x_1$};
		\draw (135-90:1.6) node{$x_2$};
		\draw (135-180:1.6) node{$x_3$};
		\draw (135-260:1.6) node{$x_4$};
\end{tikzpicture}} \right] =
 \mathbb K \left[\raisebox{-6ex}{
 	\begin{tikzpicture}[xscale=0.7,yscale=0.7]
 		\draw[line width=5pt, blue!30] (1.2,0) arc (0:360:1.2);
 		\draw[line width=2pt, blue!30] (0:1.2) to[bend left] (90:1.2) to[bend left] (180:1.2) to[bend left] (270:1.2) to[bend left] (0:1.2);
		\draw[line width=1pt] (1.2,0) arc (0:360:1.2);
		\foreach \x in{0,1,2,3}
		{
			\fill (\x*90+45:1.2) circle(0.1);
		}
		\draw (135:1.6) node{$x_1$};
		\draw (135-90:1.6) node{$x_2$};
		\draw (135-180:1.6) node{$x_3$};
		\draw (135-260:1.6) node{$x_4$};
\end{tikzpicture}} \right] 
+ 
\mathbb K \left[\raisebox{-6ex}{\begin{tikzpicture}[xscale=0.7,yscale=0.7]
		\draw[line width=5pt, blue!30] (135:1.2) arc (135:405:1.2);
		\draw[line width=5pt, red!30] (45:1.2) arc (45:135:1.2);
		\draw[dashed, thick] (45:1.2) to[bend left] (135:1.2);
		\draw[line width=2pt, blue!30] (0:1.2) --  (180:1.1) to[bend left] (270:1.1) to[bend left] (0:1.1);
		\draw[line width=1pt] (1.2,0) arc (0:360:1.2);
		\foreach \x in{0,1,2,3}
		{
			\fill (\x*90+45:1.2) circle(0.1);
		}
		\draw (135:1.6) node{$x_1$};
		\draw (135-90:1.6) node{$x_2$};
		\draw (135-180:1.6) node{$x_3$};
		\draw (135-260:1.6) node{$x_4$};
\end{tikzpicture}} \right]
+ \raisebox{-2ex}{\begin{tikzpicture}[xscale=0.7,yscale=0.7]
	\draw[-stealth'] (0,0) arc(100:420:0.5);
\end{tikzpicture}}
 \\ 
~~~~~~ ~~~~~ ~~~~ + \mathbb K \left[\raisebox{-6ex}{\begin{tikzpicture}[xscale=0.7,yscale=0.7] 
 		\draw[line width=5pt, blue!30] (45:1.2) arc (45:-135:1.2);
 		 \draw[line width=5pt, green!30] (-135:1.2) arc (-135:-135-90:1.2);
 		\draw[line width=5pt, red!30] (45:1.2) arc (45:135:1.2);
 		\draw[dashed, thick] (45:1.2) to[bend left] (135:1.2) to[bend left] (225:1.2) -- (45:1.2); 
 		\draw[line width=2pt, blue!30]  (270:1.1) to[bend left] (0:1.1);
 		\draw[line width=1pt] (1.2,0) arc (0:360:1.2);
 		\foreach \x in{0,1,2,3}
 		{
 			\fill (\x*90+45:1.2) circle(0.1);
 		}
 		\draw (135:1.6) node{$x_1$};
 		\draw (135-90:1.6) node{$x_2$};
 		\draw (135-180:1.6) node{$x_3$};
 		\draw (135-260:1.6) node{$x_4$};
 \end{tikzpicture}} \right]
 + \raisebox{-2ex}{\begin{tikzpicture}[xscale=0.7,yscale=0.7]
 		\draw[-stealth'] (0,0) arc(100:420:0.5);
 \end{tikzpicture}}  + 
\mathbb K \left[\raisebox{-6ex}{\begin{tikzpicture}[xscale=0.7,yscale=0.7] 
\draw[line width=5pt, blue!30] (45:1.2) arc (45:-45:1.2);
\draw[line width=5pt, blue!30] (135:1.2) arc (135:135+90:1.2);
\draw[line width=5pt, green!30] (45:1.2) arc (45:135:1.2);
\draw[line width=5pt, red!30] (-45:1.2) arc (-45:-135:1.2);
\draw[dashed, thick] (45:1.2) to[bend left] (135:1.2); 
\draw[dashed, thick] (-45:1.2) to[bend right] (-135:1.2);
\draw[line width=2pt, blue!30]  (180:1.1) -- (0:1.1);
\draw[line width=1pt] (1.2,0) arc (0:360:1.2);
\foreach \x in{0,1,2,3}
{
	\fill (\x*90+45:1.2) circle(0.1);
}
\draw (135:1.5) node{$x_1$};
\draw (135-90:1.5) node{$x_2$};
\draw (135-180:1.5) node{$x_3$};
\draw (135-260:1.5) node{$x_4$};
\end{tikzpicture}} \right]
+ \raisebox{-2ex}{\begin{tikzpicture}[xscale=0.7,yscale=0.7]
\draw[-stealth'] (0,0) arc(100:420:0.5);
\end{tikzpicture}}
 \\ 
+ 
\mathbb K \left[\raisebox{-6ex}{\begin{tikzpicture}[xscale=0.7,yscale=0.7] 
		\draw[line width=5pt, blue!30] (-135:1.2) arc (-135:45:1.2);
		\draw[line width=5pt, red!30] (45:1.2) arc (45:{45+180}:1.2);
		\draw[dashed, thick] (45:1.2) -- (-135:1.2);
		\draw[line width=2pt, blue!30] (270:1.1) to[bend left] (0:1.1);
		\draw[line width=2pt, red!30] (90:1.1) to[bend left] (180:1.1);
		\draw[line width=1pt] (1.2,0) arc (0:360:1.2);
		\foreach \x in{0,1,2,3}
		{
			\fill (\x*90+45:1.2) circle(0.1);
		}
		\draw (135:1.6) node{$x_1$};
		\draw (135-90:1.6) node{$x_2$};
		\draw (135-180:1.6) node{$x_3$};
		\draw (135-260:1.6) node{$x_4$};
\end{tikzpicture}} \right]
+  \raisebox{-2ex}{\begin{tikzpicture}[xscale=0.7,yscale=0.7]
		\draw[-stealth'] (0,0) arc(100:420:0.5);
\end{tikzpicture}}
+
\mathbb K \left[\raisebox{-6ex}{\begin{tikzpicture}[xscale=0.7,yscale=0.7]
		\draw[line width=5pt, blue!30] (45:1.2) arc (45:135:1.2);
		\draw[line width=5pt, blue!30] (225:1.2) arc (225:{225+90}:1.2);
		\draw[line width=5pt, red!30] (135:1.2) arc (135:{135+90}:1.2);
		\draw[line width=5pt, red!30] (-45:1.2) arc (-45:45:1.2);
		\draw[dashed, thick]  (135:1.2) to[bend left] (45+180:1.2) ;
		\draw[dashed, thick] (45+3*90:1.2) to[bend left]  (45:1.2) ;
		\draw[line width=2pt, blue!30] (90:1.1) -- (-90:1.1);
		\draw[line width=2pt, red!30] (0:1.1) -- (180:1.1);
		\draw[line width=1pt] (1.2,0) arc (0:360:1.2);
		\foreach \x in{0,1,2,3}
		{
			\fill (\x*90+45:1.2) circle(0.1);
		}
		\draw (135:1.6) node{$x_1$};
		\draw (135-90:1.6) node{$x_2$};
		\draw (135-180:1.6) node{$x_3$};
		\draw (135-260:1.6) node{$x_4$};
\end{tikzpicture}} \right]
+ 
\mathbb K \left[\raisebox{-6ex}{\begin{tikzpicture}[xscale=0.7,yscale=0.7] 
		\draw[line width=5pt, blue!30] (45:1.2) arc (45:135:1.2);
		\draw[line width=5pt, red!30] (225:1.2) arc (225:{225+90}:1.2);
		\draw[line width=5pt, green!30] (135:1.2) arc (135:{135+90}:1.2);
		\draw[line width=5pt, orange!30] (-45:1.2) arc (-45:45:1.2);
		\draw[dashed, thick] (45:1.2) to[bend left] (135:1.2) to[bend left] (45+180:1.2) to[bend left] (45+3*90:1.2) to[bend left]  (45:1.2) ;
		\draw[line width=1pt] (1.2,0) arc (0:360:1.2);
		\foreach \x in{0,1,2,3}
		{
			\fill (\x*90+45:1.2) circle(0.1);
		}
		\draw (135:1.6) node{$x_1$};
		\draw (135-90:1.6) node{$x_2$};
		\draw (135-180:1.6) node{$x_3$};
		\draw (135-260:1.6) node{$x_4$};
\end{tikzpicture}} \right]
\end{multline*}
where the symbol $\circlearrowleft$ denotes similar terms obtained by rotations of the preceding term. Each term in the expansion corresponds to a different partition of $\lbrace 1, 2,3,4\rbrace$. Since the random variables $G(i,j)$ are attached to the edges of the graph, we partition the edge set, this is why the pictures are somewhat different between Figure \ref{fig:noncrossing} and Figure \ref{fig:Kreweras}.  Then, each term in the expansion  should be understood as a product of cumulants, each cumulant corresponding to a block in the partition (i.e. a colour). To enforce $U(1)^N$ invariance, these cumulants are multiplied by (a product of) delta functions $N^{-1}\delta(x_i-x_j)$ whenever there is a dashed line joining $x_i$ and $x_j$. However, since cyclic cumulants of order $p$ are $O(N^{1-p})$ and delta terms are of order $N^{-1}$, we see that all terms in the expansion above are of order $N^{-3}$ except for the term 
$$\mathbb K \left[\raisebox{-6ex}{\begin{tikzpicture}[xscale=0.7,yscale=0.7]  
		\draw[line width=5pt, blue!30] (45:1.2) arc (45:135:1.2);
		\draw[line width=5pt, blue!30] (225:1.2) arc (225:{225+90}:1.2);
		\draw[line width=5pt, red!30] (135:1.2) arc (135:{135+90}:1.2);
		\draw[line width=5pt, red!30] (-45:1.2) arc (-45:45:1.2);
		\draw[dashed, thick]  (135:1.2) to[bend left] (45+180:1.2) ;
		\draw[dashed, thick] (45+3*90:1.2) to[bend left]  (45:1.2) ;
		\draw[line width=2pt, blue!30] (90:1.1) -- (-90:1.1);
		\draw[line width=2pt, red!30] (0:1.1) -- (180:1.1);
		\draw[line width=1pt] (1.2,0) arc (0:360:1.2);
		\foreach \x in{0,1,2,3}
		{
			\fill (\x*90+45:1.2) circle(0.1);
		}
		\draw (135:1.6) node{$x_1$};
		\draw (135-90:1.6) node{$x_2$};
		\draw (135-180:1.6) node{$x_3$};
		\draw (135-260:1.6) node{$x_4$};
\end{tikzpicture}} \right]\sim N^{-4} g_2(x_1,x_2)g_2(x_3,x_4) \delta(x_2-x_3)\delta(x_1-x_4),$$
which is subdominant. More generally, terms corresponding to partitions with crossings are always subdominant. This explains why \eqref{eq:keyrelation} holds.

The rest of the argument for Claim \ref{claim:freecumulantsg}  goes as follows \cite{hruza2022coherent}: 
\begin{enumerate}[leftmargin=2em]
	\item First one looks at the dynamical equations satisfied by expectation values of cyclic loops, $g^s(x,\tau):=\lim_{N\to\infty} N^{p-1} \mathbb E\left[G_{N^2\tau}(i_1, i_2) \ldots G_{N^2\tau}(i_p,i_1)\right]$, in the large $N$ scaling limit, with  $i_j= \lfloor Nx_j\rfloor$, $t=N^2\tau$. Here, the expectation is not taken with respect to the stationary measure, but with respect to the Brownian motions involved in the dynamics.  One can show that 
	\begin{equation}\label{eq:gs-evolution}
	 (\partial_\tau-\Delta) g^s(x,\tau) = 2\sum_{i<j } \partial_{x_i}\partial_{x_j} \left( \delta(x_i-x_j) g^s(x,\tau|\sigma_i)g^s(x,\tau|\sigma_j) \right),
	 \end{equation}
	where $\Delta:=\sum_{j=1}^p \partial_{x_j}^2$ and $g^s(x,\tau|\sigma_i)$ (resp. $g^s(x,\tau|\sigma_j)$) are the expectation values of the smaller loops $\sigma_i$, containing $x_i$ but not $x_j$, (resp. $\sigma_j$, containing $x_j$ but not $x_i$) obtained by splitting the initial loop as follows  
		\begin{center}
			\begin{tikzpicture}[xscale=0.7,yscale=0.7]  
				\draw[red!50, line width=2pt] (36:2) -- (0:2) -- (-36:2) -- (-2*36:2) -- (-3*36:2) -- (-4*36:2) -- cycle;
				\draw[red!50, line width=2pt] (5*36:2) -- (4*36:2) -- (3*36:2) -- (2*36:2) -- cycle;
				\draw[thick] (2,0) arc (0:360:2);
				\foreach \x in {1,2,...,10}
				{\fill (\x*36:2) circle(0.1);}
				\draw (36:2.5) node{$x_j$}; 
				\draw (0:2.8) node{$x_{j+1}$}; 
				\draw (-5*36:2.5) node{$x_{i}$}; 
				\draw (-4*36:2.5) node{$x_{i-1}$}; 
				\draw (4*36:2.5) node{$x_{i+1}$};
				\draw (3*36:2.5) node[rotate=15]{\ldots}; 
				\draw (-36:2.5) ; 
				\draw (-2*36:2.5) node[rotate=20]{\ldots};     
				\draw (2*36:2.5) node{$x_{j-1}$};  
				\draw[red] (3*36+18:1.5) node{$\sigma_i$};
				\draw[red] (-36-18:0.8) node{$\sigma_j$};
				\end{tikzpicture}
		\end{center}
The evolution for these $g^s$ has therefore a triangular structure.

	\item Given the expectation values $g^s(x,\tau)$, one defines the cyclic cumulants $g_p(x,\tau)$ using  \eqref{eq:keyrelation} and $g_\pi(x,\tau)=\prod_{B\in\pi}g_{|B|}(x,\tau)$  as in \eqref{eq:defgpi} for any non-crossing partition. Then, mimicking the moment-cumulant formula \eqref{eq:deffreecumulants} in the free probability context, one defines 
	$$ \varphi(x,\tau) := \sum_{\pi\in \mathrm{NC}(\mathcal P_p)} g_\pi(x,\tau).$$
	Compare to \eqref{eq:keyrelation}, this formula does not contain the Dirac delta-function associated to the Kreweras complement of $\pi$. Then, 
	one shows using the type of reasoning above and (combinatorial) properties of non-crossing partitions that we have 
	\begin{equation}
		(\partial_{\tau}-\Delta) \varphi(x,\tau) = 2\sum_{i<j } \delta(x_i-x_j) (\partial_{x_i}\varphi(x,\tau|\sigma_i))(\partial_{x_j}\varphi(x,\tau|\sigma_j)),
		\label{eq:EDPvarphi}
	\end{equation}
	where $\varphi(x,\tau|\sigma_i)$ (resp. $\varphi(x,\tau|\sigma_j)$) is defined as above by breaking the initial loop in two pieces $\sigma_i$ and $\sigma_j$ and restricting $\varphi(x,\tau)$ to the smaller loop $\sigma_i$ (resp. $\sigma_j$).
	\item Finally, one checks that $\varphi(x) = \min\lbrace x_1, \dots, x_p\rbrace$ is the stationary solution of \eqref{eq:EDPvarphi} with appropriate boundary conditions. 
\end{enumerate} 
This implies that $g_\pi(x)$ in the stationary measure are the free cumulants associated to moments given by the function $\varphi(x)$, that is the free cumulants of the variables $I_{x}$ as in Claim \ref{claim:freecumulantsg}.

\begin{remark}
The proof of the Claim \ref{claim:freecumulantsg} in \cite{biane2021combinatorics} is different from the argument above. It starts from the stationarity conditions for the cumulants of cyclic loops \cite{bernard2019open}. To write them we need an alternative notation.  Until now we ordered the points on the loops. Instead we could order them on the interval $[0,1]$, so that $ 0\leq x_1 < x_2 < \dots < x_p\leq 1 $ and consider cumulants (with respect to the stationary measure for the process $G_t$)
$$ g_{\sigma}(x) = \mathbb K
\left[ \raisebox{-8.5ex}{\begin{tikzpicture}[xscale=0.8,yscale=0.8]
\draw[thick] (1.2,0) arc (0:360:1.2);
\foreach \x in {1,2,...,10}
{\fill (\x*36:1.2) circle(0.1);}
\foreach \x in {1,2,...,4}
\draw ({-\x*36+4*36}:1.7) node{\footnotesize $x_{\sigma^{\x}(1)}$};
\draw (4*36:1.7) node{$x_{1}$};
\draw ({5*36}:1.6) node[rotate=90]{$\ldots$};
\draw ({6*36}:1.6) node[rotate=-50]{$\ldots$};
\draw ({7*36}:1.6) node[rotate=-15]{$\ldots$};
\draw ({8*36}:1.6) node[rotate=20]{$\ldots$};
\draw ({9*36}:1.6) node[rotate=55]{$\ldots$};
\end{tikzpicture}} \right], $$
for any cyclic permutation $\sigma \in \mathcal S_p$.
Those cyclic cumulants are polynomials  \cite{bernard2019open}. Given $\sigma \in \mathcal S_p$, and for any pair of adjacent variables $x_j$ and $x_{j+1}$, $g_{\sigma}(x)$ can be decomposed as 
$$ g_{\sigma}(x) = A_j(\sigma) + B_j(\sigma)x_j + C_j(\sigma) x_{j+1} + D_j(\sigma) x_jx_{j+1}, $$
where the coefficients $A_j,B_j,C_j,D_j$ may depend on all variables $x_k$ but not on $x_j$ and $x_{j+1}$. The stationarity conditions on $g_\sigma$ then translate into conditions on those coefficients~\cite{bernard2021solution}
\begin{align*}
A_j(\tau_j\circ\sigma) &= A_j(\sigma),\\
B_j(\tau_j\circ\sigma) &= C_j(\sigma) + \partial_{x_j}g_{\sigma_j^-}(x)\,\partial_{x_{j+1}}g_{\sigma_j^+}(x), \\
C_j(\tau_j\circ\sigma) &= B_j(\sigma) -\partial_{x_j}g_{\sigma_j^-}(x)\, \partial_{x_{j+1}}g_{\sigma_j^+}(x), \\
D_j(\tau_j\circ\sigma) &= D_j(\sigma), 
\end{align*}
where $\tau_j\in \mathcal S_p$ is the transposition exchanging $j$ and $j+1$,  and $\tau_j\circ\sigma$ denotes the composition of the two permutations $\tau_j$ and $\sigma$. Similarly as in \eqref{eq:gs-evolution}, the cyclic permutations $\sigma_j^\pm$ are defined by breaking the initial cyclic permutation $\sigma$ into two smaller cyclic permutations. Namely, let $r,s$ be such that $\sigma^{r}(j)=j+1$ and $\sigma^s(j+1)=j$. Then, the cyclic permutations $\sigma_j^- $ and $\sigma_j^+$ are defined in such a way that the cyclic order induced by $\sigma$ in the $r$ variables $x_j, x_{\sigma(j)}, \dots , x_{\sigma^{r-1}(j)}$ is encoded by $\sigma_j^-$ while the cyclic order induced by $\sigma$ in the $s$ variables $x_{j+1}, x_{\sigma(j+1)}, \dots, x_{\sigma^{s-1}(j+1)}$ is encoded by $\sigma_j^+$. Pictorially, 
		\begin{center}
			\begin{tikzpicture}[xscale=0.75,yscale=0.75]
				\draw[red!50, line width=2pt] (36:2) -- (0:2) -- (-36:2) -- (-2*36:2) -- (-3*36:2) -- (-4*36:2) -- cycle;
				\draw[red!50, line width=2pt] (5*36:2) -- (4*36:2) -- (3*36:2) -- (2*36:2) -- cycle;
				\draw[thick] (2,0) arc (0:360:2);
				\foreach \x in {1,2,...,10}
				{\fill (\x*36:2) circle(0.1);}
				\draw (36:2.5) node{$x_j$}; 
				\draw (0:2.5) node{$x_{\sigma(j)}$}; 
				\draw (-5*36:2.5) node{$x_{j+1}$}; 
				\draw (-4*36:2.5) node{$x_{\sigma^{r-1}(j)}$}; 
				\draw (4*36:2.5) node{$x_{\sigma(j+1)}$};
				\draw (3*36:2.5) node[rotate=15]{\ldots}; 
				\draw (-36:2.5) node{$x_{\sigma^2(j)}$}; 
				\draw (-2*36:2.5) node[rotate=20]{\ldots};     
				\draw (2*36:2.5) node{$x_{\sigma^{s-1}(j+1)}$};  
				\draw[red] (3*36+18:1.5) node{$\sigma_j^+$};
				\draw[red] (-36-18:0.8) node{$\sigma_j^-$};
				\end{tikzpicture}
		\end{center}
The above equations admit a unique solution. The proof in \cite{biane2021combinatorics} consists in solving them combinatorially, and understanding their free probability interpretation.
\end{remark}

\section{Ensembles of structured random matrices}	
	\label{sec:structured}
	
	 In this section, following \cite{bernard2023structured,bernard2025addition} we describe a possible general framework for dealing with a large class of structured random matrices. In random matrix theory, the probability distribution of matrix entries is often invariant under permutations of the indices. This is the case of Wigner matrices for instance.  Another example which is important for us  is a Haar distributed orbit $M=UDU^*$, with $D$ some fixed diagonal matrix and  $U$ a Haar distributed unitary matrix. In both cases, the law of $M$ is $U(N)$ invariant, so it is a fortiori permutation invariant. This implies that quantities such as 
	$$ \mathbb E\left[ M_{i_1,j_1} \cdots M_{i_p,j_p}\right] $$ 
	depend on the indices only through the topology of the associated graph. For example, in the case of a random matrix $M=UDU^*$, the expectation values of cyclic loops, with all indices $i_1, \dots, i_p$ distinct, behave as 
	$$ \mathbb E\left[ M_{i_1,i_2} \cdots M_{i_p,i_1}\right]  = \frac{1}{N^{p-1}} \kappa_p(\mu_D) + O(N^{-p}),$$
	where $\kappa_p(\mu_D) $ denote the free cumulants of the limiting spectral distribution of the diagonal matrix $D$, as the size $N$ goes to infinity. This expression, and in particular its scaling, resembles those encountered in Section \ref{sec:openQSSEP}, except that in QSSEP those expectation values depend explicitly on the points $x_k=i_k/N$.
We are interested in this Section in a generalization of these ensembles of random matrices for \emph{structured} matrices whose law is not invariant under permutations. 
	 Similar random matrix ensembles are considered in the context of the Eigenstate Thermalization Hypothesis \cite{foini2018eigenstate, pappalardi2022eigenstate}. The class of random matrices we consider is however different from \cite{handel2016structured} which  also employs the phrase ``structured random matrices''.
	\subsection{Cyclic cumulants and free probability}
	Let us describe the ensembles of random matrices introduced in \cite{bernard2023structured}. 
	\begin{assumption}
				\label{assumption}
				 Let $\mathbb P_N$  be a sequence of probability measures on $N\times N$ matrices $M_N$ satisfying the following properties (we will often drop the subscript $N$ and write simply the measure $\mathbb P$ and  the matrix $M$): 
		\begin{enumerate}[leftmargin=2em]
			\item ($U(1)^N$ invariance) For any diagonal matrix $u\in U(1)^N$, i.e. $u=\mathrm{diag}(e^{\I\theta_1}, \dots, e^{\I\theta_N})$, $\theta_j\in \R$, 
			$$ M\overset{(d)}{=} uMu^*.$$
			\item (Scaling of cyclic cumulants) For all $p\geq 1$, there exists a continuous  function $g_p(x_1, \dots, x_p)$ such that if $i_j=i_j(N)$ are distinct integers chosen so that $i_j(N)/N$ converges to $x_j$, 
			\begin{equation} \label{ex:def-gp}
			  \mathbb E\left[M_{i_1,i_2} \cdots  M_{i_p,i_1} \right] = \frac{1}{N^{p-1}} g_p(x_1, \dots, x_p) + O\left(\frac{1}{N^p}\right).
			  \end{equation}
			  \item (Pinching of loops) If in the sequences $(i_1,\cdots,i_p)$ some of the indices coincide, e.g. $i_n(N)= i_m(N)$ for $n\not= m$, then
						$$\lim_{N\to\infty} N^{p-1} \mathbb K\left[M_{i_1,i_2}, \ldots ,M_{i_p,i_1} \right]= g_p(x),$$
			\item (Scaling of disconnected cycles) The cumulant associated to $r$ disjoint cycles with a total of $n$ matrix elements is of order $O(N^{2-r-n})$. More precisely, using the graphical notations that we have set up in Section \ref{sec:openQSSEP}, 
\begin{equation}
	\mathbb K\left[ \raisebox{-3ex}{\begin{tikzpicture}[yscale=0.5, xscale=0.5] 
			\draw[thick] (1.2,0) arc (0:360:1.2);
			\foreach \x in{0,1,2,3}
			{
				\fill (\x*90+45:1.2) circle(0.15);
			}
			\begin{scope}[xshift=3cm]
				\draw[thick] (1.2,0) arc (0:360:1.2);
				\foreach \x in{0,1,...,5}
				{
					\fill (\x*360/5+45:1.2) circle(0.15);
				}
			\end{scope}
			\begin{scope}[xshift=6cm]
				\draw(0,0) node{$\dots$};
			\end{scope}
			\begin{scope}[xshift=9cm]
				\draw[thick] (1.2,0) arc (0:360:1.2);
				\foreach \x in{0,1,2,3}
				{
					\fill (\x*360/3+45:1.2) circle(0.15);
				}
			\end{scope}
	\end{tikzpicture}}  \right]  = O(N^{2-r-n}),
\label{eq:scalingdisjointcycles}
\end{equation} 
in the sense that there exists a constant $C_{n,r}$ such that all such cumulants are bounded by  $C_{n,r}N^{2-r-n}$.
		\end{enumerate} 
	\end{assumption} 

The property \textit{(4)} is stronger than the corresponding property in the original reference \cite{bernard2023structured}, this refined version was introduced in the addendum  \cite{bernard2025addition}.	In \textit{(2)}, because of $U(1)^N$ invariance, we can replace the expectation values by the cumulants since we assume that all indices are distinct. Thus, 
using the graphical notation, the functions $g_p$ defined in \eqref{ex:def-gp} are the scaled cumulants of cyclic loops, 
\begin{equation} \label{eq:local-cumulant}
	g_p(x_1, \dots, x_p) =\lim_{N\to\infty} N^{1-p} \, \mathbb K
	\left[ \raisebox{-6.5ex}{\begin{tikzpicture}[xscale=0.7,yscale=0.7]
			\draw[thick] (1,0) arc (0:360:1);
			\foreach \x in {1,2,...,10}
			{\fill (\x*36:1) circle(0.1);}
			\foreach \x in {1,2,...,4}
			\draw ({-\x*36+4*36}:1.4) node{$x_{\x}$};
			\draw ({6*36}:1.4) node[rotate=-50]{$\ldots$};
			\draw ({7*36}:1.4) node[rotate=-15]{$\ldots$};
			\draw ({8*36}:1.4) node[rotate=20]{$\ldots$};
			\draw ({9*36}:1.4) node[rotate=55]{$\ldots$};
			\draw ({5*36}:1.4) node[rotate=90]{$\ldots$};
			\draw (4*36:1.4) node{$x_p$};
	\end{tikzpicture}} \right]. 
	\end{equation}
These $g_p$ are called \emph{local free cumulants}, since they have a natural interpretation in conditioned free probability (see Remark \ref{rem:freeconditioned} below).

The general approach to structured random matrices developed below is based on the  method of moments. This is why  properties \textit{(2)}, \textit{(3)} and \textit{(4)} in Assumptions \ref{assumption} concern moments and cumulants. Consequently, a number of statements have to be interpreted as formal power series, as we stress below. In random matrix theory, it is often necessary to impose some additional assumptions on the operator norm of the involved matrices. Estimating the operator norm of matrices from minimal assumptions is typically a difficult problem \cite{haagerup2005new}  -- see also \cite{hayes2020random, belinschi2022strong, bordenave2023norm, vanhandel2025strong} for recent developments around the Peterson-Thom conjecture. In these notes, we do not discuss such  assumptions.

\begin{remark}
Since these properties are formulated in terms of cumulants, they are preserved by shifting the matrices $M$ by diagonal matrices, $M_{ij}\to M_{ij}-\delta_{ij} a_i$. Similarly these properties are preserved by multiplying $M$ by a diagonal matrix, on the left or on the right, that is under transformation $M\to \Delta_L M \Delta_R$ with $\Delta_{L}, \Delta_R$ being diagonal matrices.
\end{remark}

\begin{remark}
As already mentioned, Assumptions \ref{assumption} are fulfilled by Wigner matrices (possibly with a variance profile \cite{girko2012theory, shlyakhtenko1996random, guionnet2000large}) or Haar distributed orbits. It is believed that they are also satisfied in QSSEP, although a proof is missing (the proof of the scaling property \textit{(4)} is lacking). 
\end{remark}	

\begin{remark} Defining an ensemble of structured random matrices requires the data of its local free cumulants \eqref{eq:local-cumulant}. It is an apparently difficult problem to characterize which set of functions $g_p$ forms a possible set of local free cumulants. 
\end{remark}

	\begin{example}
The scalings \eqref{eq:scalingdisjointcycles} imply for example that 
$$  \mathbb K \left[\raisebox{-5ex}{\begin{tikzpicture}[yscale=0.5, xscale=0.5] 
		\draw[thick] (0,0) arc (-90:270:1);
		\draw[thick] (3,0) arc (-90:270:1);
		\fill (0,0) circle(0.2);
		\fill (3,0) circle(0.2);
		\draw (0,0) node[below]{$i$};
		\draw (3,0) node[below]{$j$};
\end{tikzpicture}} \right]  = O\left(N^{2-2-2}\right) = O\left( N^{-2} \right) $$
but 
$$ \mathbb K \left[\raisebox{-2ex}{\begin{tikzpicture}[yscale=0.5, xscale=0.5] 
		\draw[thick] (1,0) arc (0:360:1);
		\fill (0:1) circle(0.2);
		\fill (180:1) circle(0.2);
		\node[left] at (180:1) {$i$};
		\node[right] at (0:1) {$j$};
\end{tikzpicture}} \right]  = O\left( N^{2-1-2}\right) = O\left(N^{-1}\right).$$
Comparing with Section \ref{sec:openQSSEP}, we see that these scaling properties are satisfied in QSSEP.
\end{example}

	More generally, the scalings in Assumptions \ref{assumption} imply that $\tr(M)$ concentrates fast around its mean:  we have that (as a formal power series in $z$)
	\begin{equation}
		\lim_{N\to\infty} N^{-1} \log \mathbb E\left[ e^{z \tr(M)} \right] = \lim_{N\to\infty} N^{-1} z \mathbb E\left[\tr(M)\right],
		\label{eq:selfaveraging}
	\end{equation} 
in the sense that  all cumulants are negligible compared to the first one as $N$ goes to infinity.
Indeed, let us consider the first two cumulants. The first one is 
	$$ \mathbb K\left[ \tr(M)\right] = \sum_{i=1}^N \mathbb K\left[ \raisebox{-4ex}{\begin{tikzpicture}[scale=0.5] 
			\draw[thick] (0,0) arc (-90:270:1);
			\fill (0,0) circle(0.2);
			\node[below] at (0,0) {$i$};
	\end{tikzpicture}}\right]  \sim N \int_0^1 \mathrm dx\, g_1(x),$$
while the second cumulant is 
\begin{align*}
				\mathbb K\left[ \tr(M),\tr(M)\right] &= \sum_{1\leq i\neq j\leq N} \mathbb K\left[\raisebox{-5ex}{\begin{tikzpicture}[yscale=0.5, xscale=0.5] 
						\draw[thick] (0,0) arc (-90:270:1);
						\draw[thick] (3,0) arc (-90:270:1);
						\fill (0,0) circle(0.2);
						\fill (3,0) circle(0.2);
						\draw (0,0) node[below]{$i$};
						\draw (3,0) node[below]{$j$};
				\end{tikzpicture}} \right] + \sum_{i=1}^N \;\;\; \mathbb K\left[ \raisebox{-5ex}{\begin{tikzpicture}[scale=0.5] 
				\draw[thick] (0,0) arc (-90:270:1);
				\draw[thick] (0,0) arc (90:-270:1);
				\fill (0,0) circle(0.2);
			\end{tikzpicture}}  \right]\\ 
				&=N^2 \times O\left(N^{-2}\right) + N \times O\left(N^{-1}\right) = O(1).			
\end{align*}
It is easy to check using Assumptions \ref{assumption} that higher cumulants are sub-leading in $1/N$. For such matrices, non-crossing partitions play a role similar as in QSSEP. We have:

\begin{claim}[\cite{bernard2023structured}]
Let $M=M_N$ be a sequence of random matrices  with law $\mathbb P_N$ satisfying Assumptions \ref{assumption}. Let $\Delta_1, \dots, \Delta_p$ be diagonal matrices such that $(\Delta_k)_{ii} = \psi_k(i/N)$ for some continuous functions $\psi_k$. Then, 
\begin{equation}
 \lim_{N\to\infty} N^{-1} \mathbb E\left[ \tr(M\Delta_1M\Delta_2 \cdots M\Delta_p)  \right] = \int_{[0,1]^p} T_p(x_1, \dots, x_p)\prod_{j=1}^p\mathrm dx_j \psi_j(x_j) ,
 \end{equation}
where 
\begin{equation} \label{eq:gpi-deltapi*}
 T_p(x_1, \dots, x_p) = \sum_{\pi\in \mathrm{NC}(\mathcal P_p)} g_{\pi}(x) \delta_{\pi^*}(x) ,
 \end{equation}
and $g_{\pi}(x)$ and $  \delta_{\pi^*}(x)$ are defined as in \eqref{eq:keyrelation}.
\label{claim:traces}
\end{claim}
This is obtained via the same argument as in Section \ref{sec:openQSSEP} (which we have sketched only for $p=4$). 

\begin{remark} \label{rem:freeconditioned}
This formula illustrates the fact that the notion of conditional probability, in a non-commutative setup, is a well-suited framework to study structured random matrices. 
Let $\mathcal D \subset \mathcal A$ be a unital subalgebra. Following \cite{Mingo2017Free}, a \emph{conditional expectation value} is a map $E^\mathcal{D}:\mathcal A \to \mathcal D$ such that  $E^\mathcal{D}[1]=1$  and $E^\mathcal{D}[\Delta a \Delta']=\Delta E^\mathcal{D}[a]\Delta'$ for all $a\in \mathcal A$ and $\Delta, \Delta'\in \mathcal D$. Furthermore, the \emph{operator-valued distribution} of a random variable $a\in \mathcal A$ is given by all \emph{operator-valued} moments $E^\mathcal{D}[a\Delta_1a\cdots a\Delta_{n-1}a]\in\mathcal D$ where $\Delta_1,\cdots,\Delta_{n-1}\in\mathcal D$.

In the present context, we should consider the case where $\mathcal A$ is the ensemble of $N\times N$ random matrices $M$ satisfying assumptions {\it (1)-(3)}, and $\mathcal D$ is the subalgebra of deterministic (bounded) diagonal matrices. In the large $N$ limit the algebra $\mathcal D$ is identified with $L^\infty[0,1]$. Furthermore, for $M\in \mathcal A$ we define the conditional expectation value to be 
\begin{equation*} 
E^\mathcal{D}[M]:=\mathrm{diag}(\mathbb E[M_{11}],\cdots,\mathbb E[M_{NN}]),
\end{equation*}
that is, one takes the usual expectation value of the matrix elements and sets all non-diagonal elements to zero. Since we are only interested in this concrete example, we will always denote elements of $\mathcal A$ by $M$ in the following definitions (instead of $a$).

As in the scalar case, the operator-valued free cumulants are defined via a moment-cumulant relation  \cite{Mingo2017Free}: 
The \emph{$\mathcal D$-valued free cumulants} $\kappa^\mathcal{D}_n:\mathcal A^n\to\mathcal{D}$ are defined through the $\mathcal{D}$-valued moments by 
	\begin{equation*}
		E^\mathcal{D}[M_1\cdots M_n]=\sum_{\pi\in \mathrm{NC}(\mathcal{P}_n)} \kappa_\pi^\mathcal{D}(M_1,\cdots,M_n),
	\end{equation*}
	where the $\kappa_\pi^\mathcal{D}$ are constructed from the family of linear functions $\kappa_n^\mathcal{D}:=\kappa_{1_n}^\mathcal{D}$ respecting the nested structure of the parts appearing in $\pi$ -- see Figure \ref{fig:nested}. 
	
	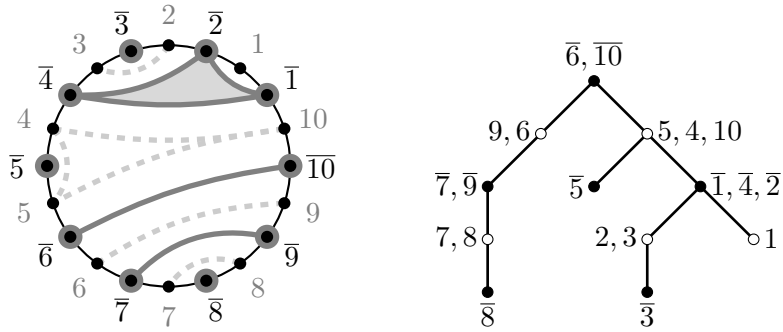
\begin{figure}
		\begin{tikzpicture}[scale=0.8]
			\draw[thick] (2,0) arc(0:360:2);
			\filldraw[line width=2pt, draw= black!50, fill=gray!30] (36:2) to[bend left] (2*36:2) to[bend left=20] (4*36:2) to[bend right=10] (36:2);
			\fill[black!50] (36:2) circle(0.2);
			\fill[black!50] (2*36:2) circle(0.2);
			\fill[black!50] (4*36:2) circle(0.2);
			\draw[line width=2pt, black!50] (-36:2) to[bend right] (-3*36:2);
			\fill[black!50] (-36:2) circle(0.2);
			\fill[black!50] (-3*36:2) circle(0.2);
			\draw[line width=2pt, black!50] (-4*36:2) to[bend left=10] (0:2);
			\fill[black!50] (-4*36:2) circle(0.2);
			\fill[black!50] (0:2) circle(0.2);
			\fill[black!50] (-2*36:2) circle(0.2);
			\fill[black!50] (3*36:2) circle(0.2);
			\fill[black!50] (5*36:2) circle(0.2);
			\draw[dashed, line width=2pt, gray!40] (2*36+18:2) to[bend left=40] (3*36+18:2);
			\draw[dashed, line width=2pt, gray!40] (4*36+18:2) to[bend left=40] (5*36+18:2);
			\draw[dashed, line width=2pt, gray!40] (0*36+18:2) to[bend left=10] (4*36+18:2);
			\draw[dashed, line width=2pt, gray!40] (0*36+18:2) to[bend right=10] (5*36+18:2);
			\draw[dashed, line width=2pt, gray!40] (-1*36+18:2) to[bend right=15] (-4*36+18:2);
			\draw[dashed, line width=2pt, gray!40] (-2*36+18:2) to[bend right=40] (-3*36+18:2);
			\foreach \x in {1,...,10}
			{\fill (\x*36:2) circle(0.1);
				\draw (\x*36:2.5) node{$\overline{\x}$};}
			\foreach \x in {1,...,10}
			{\fill (\x*36+18:2) circle(0.1);
				\draw[gray] (\x*36+18:2.5) node{$\x$};}
		\end{tikzpicture}
		\hspace{0.8cm} 
		\begin{tikzpicture}[scale=0.7]
\draw[line width=1pt] (0,0) -- (-1,-1) -- (-2,-2) -- (-2,-4);
\draw[line width=1pt] (0,0) -- (1,-1) -- (2,-2) -- (3,-3);
\draw[line width=1pt]  (1,-1) -- (0,-2) ; 
\draw[line width=1pt]  (2,-2) -- (1,-3) -- (1,-4); 
\draw[fill=black, draw=black] (0,0) circle(0.1);
\draw[fill=white, draw=black] (-1,-1) circle(0.1);
\draw[fill=white, draw=black] (1,-1) circle(0.1);
\draw[fill=black, draw=black] (-2,-2) circle(0.1);
\draw[fill=black, draw=black] (0,-2) circle(0.1);
\draw[fill=black, draw=black] (2,-2) circle(0.1);
\draw[fill=white, draw=black] (-2,-3) circle(0.1);
\draw[fill=white, draw=black] (1,-3) circle(0.1);
\draw[fill=white, draw=black] (3,-3) circle(0.1);
\draw[fill=black, draw=black] (-2,-4) circle(0.1);
\draw[fill=black, draw=black] (1,-4) circle(0.1);
\draw[fill=black, draw=black] (0,0) node[above]{$\overline 6,\overline{10}$};
\draw[fill=white, draw=black] (-1,-1) node[left]{$9,6$};
\draw[fill=white, draw=black] (1,-1) node[right]{$5,4,10$};
\draw[fill=black, draw=black] (-2,-2) node[left]{$\overline 7, \overline 9$};
\draw[fill=black, draw=black] (0,-2) node[left]{$\overline 5$};
\draw[fill=black, draw=black] (2,-2) node[right]{$\overline 1, \overline 4, \overline 2$};
\draw[fill=white, draw=black] (-2,-3) node[left]{$7,8$};
\draw[fill=white, draw=black] (1,-3) node[left]{$2,3$};
\draw[fill=white, draw=black] (3,-3) node[right]{$1$};
\draw[fill=black, draw=black] (-2,-4) node[below]{$\overline 8$};
\draw[fill=black, draw=black] (1,-4) node[below]{$\overline 3$};

		\end{tikzpicture}
		\caption{Non-crossing partitions have a  nested structure: for instance in the present example, the part $\{\bar 8\}$ is nested inside the part $\{\bar 7,\bar 9\}$ which is itself nested in $\{\bar 6,\bar 10\}$. One associates to this example the planar tree shown on the right, where the black vertices correspond to the blocks of $\pi$ and the white vertices correspond to the blocks of the Kreweras complement $\pi^*$. 
		} 
		\label{fig:nested}
	\end{figure}
It turns out -- see \cite{bernard2023structured} for details -- that  the cumulants of cyclic loops $g_\pi(x)$ appearing in \eqref{eq:gpi-deltapi*} can be identified with the operator-valued free cumulants, that is:
	\begin{multline}\label{eq:op-val-cum_pi}
		\kappa_\pi^\mathcal{D}(M\Delta_1,\cdots,M\Delta_n)(x)  \\  = \int \left( \prod_{k=1}^{n-1} \mathrm dx_k\psi_k(x_k)\right)\,\psi_n(x) g_\pi(x_1,\cdots,x_{n-1},x) \delta_{\pi^*}(x_1,\cdots,x_{n-1},x), 
	\end{multline}
for $\Delta_k$ diagonal matrices with entries $(\Delta_k)_{ii}=\psi_k(i/N)$ with $\psi_k$ bounded functions on $[0,1]$.
\end{remark}

	\subsection{Stability under non-linear transformation}
The class of matrices satisfying Assumptions \ref{assumption} is stable under certain non-linear transformations. 
\begin{claim}[\cite{bernard2023structured,bernard2025addition}] \label{claim:polynomialstability}
	Let $M=M_N$ satisfy Assumptions \ref{assumption}. Then for every polynomial $P$, the sequence of random matrices $P(M)$ also satisfies Assumptions \ref{assumption}. 
\end{claim}
The matrix $P(M)$ has new explicit local free cumulants, which can be expressed in terms of those of $M$. As a consequence of this stability property, we also obtain that $\tr(P(M))$ concentrates to its expectation in the sense of \eqref{eq:selfaveraging}, namely
\begin{equation} \label{eq:PM-self}
 \lim_{N\to\infty} N^{-1} \log \mathbb E\left[ e^{z \tr P(M)} \right] =\lim_{N\to\infty} z N^{-1}\mathbb E\left[\tr P(M)\right], 
 \end{equation}
for any polynomial $P$.

\begin{claim}[\cite{bernard2023structured,bernard2025addition}] \label{claim:entrywisetransformstability}
Let $M=M_N$ satisfy Assumptions \ref{assumption}. Assume that they are centred, $\mathbb{E}[M_{ii}]=0$. Define a matrix $Y$ with coefficients 
$$ Y_{i,j} = M_{i,j} f_{i,j}^{(N)}(N \vert M_{i,j}\vert^2) $$
where the $f^{(N)}_{i,j}$ are polynomials
$$f^{(N)}_{i,j}(u) = \sum_{k=0}^{d} a_k\left(\frac{i}{N},\frac{j}{N}\right) u^k$$  
where the functions $a_k(x,y)$ are continuous in $(x,y)$. Then the sequence of matrices $Y$ satisfies Assumptions \ref{assumption}.    \label{claim:weirdnonlinear}
\end{claim}

The property {\it (4)} in Assumptions  \ref{assumption} was reinforced in \cite{bernard2025addition}, and arguments for its stability under the above non-linear transformations were given there. However, as discussed in \cite{bernard2025addition}, these arguments are not enough to complete the proof of Claim \ref{claim:polynomialstability} and Claim \ref{claim:entrywisetransformstability}.
The type of entrywise  non-linear transformations considered in Claim \ref{claim:weirdnonlinear} is a type of transformations considered in machine learning  \cite{Speicher2023RMML, speicher2024entrywise}.

Let us now mention briefly two applications of the previous claims, one in random matrix theory and another in quantum stochastic processes.	

	\subsection{Application: Spectrum of sub-matrices} 
	Let $M=M_N$ be (a sequence of) random matrices satisfying Assumptions \ref{assumption} and consider $M_{\rm red}$ a $\ell N\times\ell N$ sub-matrix of $M$ (we take the same subset of rows and columns, this is sometimes called a principal submatrix). What can be said about the law of $M_{\rm red}$? If $M$ is a Wigner matrix, it is clear that $M_{\rm red}$ is still a Wigner matrix. Let us consider the slightly more involved case of a Haar-distributed orbit $M=UDU^*$. How are sub-matrices distributed? Asymptotically, they are distributed as a Haar orbit with spectral measure $\mu_{D}^{(\ell)}$, the free compression of $\mu_D$ obtained  by scaling all the free cumulants of $\mu_D$ by $1/\ell$ (see more details in Example \ref{ex:freecompression}). Obtaining this answer using the present framework is not more difficult than computing the spectrum of the full matrix, thanks  to the stability of Assumptions \ref{assumption} under left/right multiplication by diagonal matrices. More generally, we have the following.

\begin{claim}[\cite{bernard2023structured}] \label{claim:subblock} 
Let $M=M_N$  satisfy Assumptions \ref{assumption}. 	Let $H=H_N$ be a sequence of diagonal matrices such that $H_{ii} = h(i/N)$ where $h$ is a continuous function,  and set $M_h=H^{1/2}M H^{1/2}$. Then, the spectrum of   $M_h$ is characterized by  the formal power series in $1/z$,
$$ \mathcal F(h; z) = \lim_{N\to\infty} \frac{1}{N}\mathbb E  \left[ \tr(\log(1-z^{-1}M_h))\right],$$
which is given explicitly  by 
\begin{equation}
	\mathcal F(h; z) =\int_0^1 \mathrm dx\, \left[ \log\left(1-z^{-1}h(x)b(x)\right) + a(x)b(x)\right] - \mathcal F_0(a) ,
	\label{eq:extremization}
\end{equation} 
where $a(x) $ and $b(x)$ solve
\begin{equation}
	a(x)  = \frac{h(x)}{z-h(x)b(x)}, \;\; b(x) = \frac{\delta \mathcal F_0(a)}{\delta a(x)}, 
	\label{eq:saddle}
\end{equation}
and 
$$ \mathcal F_0(a) := \sum_{p\geq 1 } \frac{1}{p}\int_{[0,1]^p} g_p(x_1, \dots, x_p)\prod_{j=1}^p\mathrm dx_j \, a(x_j).$$
The relations \eqref{eq:saddle} are the extremization condition for the r.h.s. of \eqref{eq:extremization} viewed as functional of $a$ and $b$. 
\end{claim}
The functional derivative in \eqref{eq:saddle} is given by 
$$ \frac{\delta \mathcal F_0(a)}{\delta a(x)} = \sum_{p\geq 1 } \frac{1}{p} \sum_{j=1}^p \int_{[0,1]^{p-1}} g_p(x_1, \dots ,  x_p)\Big\vert_{x_j=x}\prod_{i\neq j} \mathrm dx_i \, a(x_i).$$

Several derivations of this claim were presented in \cite{bernard2023structured}. All these  arguments rely on the tree structure underlying free probability, which is induced by the nested structure of non-crossing partitions as in Figure \ref{fig:Kreweras}. The most conceptual route is probably the one using conditional non-commutative probability and the relation between the cyclic loop cumulants and operator valued free cumulants. The saddle point equations \eqref{eq:saddle} are then equivalent to the relation between the operator valued Cauchy and $R$-transforms in free probability  \cite[Chap. 9, Theorem 11]{Mingo2017Free}. 
In order to make this argument fully rigorous, the first step would be to prove the relation \eqref{eq:op-val-cum_pi} between the $g_{\pi}$ defined by  \eqref{eq:defgpi} and the free cumulants of \cite{Mingo2017Free}.

  In the case of Wigner matrices with variance profile, the equations for the eigenvalue density in \cite{casati1993generalized, shlyakhtenko1996random} are equivalent to \eqref{eq:saddle} with only $g_1$ and $g_2$  non vanishing. 

\begin{remark} The large $N$ limit of the spectrum of $M_h$ is determined by the limit of its Stieltjes transform $\mathcal{G}(z):=\lim_{N\to\infty }N^{-1} \tr\left(\frac{1}{z-M_h}\right)$. Under Assumptions \ref{assumption}, using \eqref{eq:PM-self}, the Stieltjes transform is equal to its mean in the large $N$ limit.
By Claim \ref{claim:subblock}, the latter is simply obtained by taking a derivative with respect to $z$ in the solution \eqref{eq:extremization}, i.e. 
$$ 
\mathcal{G}(z) = \int_0^1\! \frac{dx}{z-h(x)b(x)} 
$$
where $b$ is given by \eqref{eq:saddle}. In order to get information on the sub-matrix of $M$ made by selecting the row/columns corresponding to an interval $I\subset[0,1]$, it suffices to apply the formula above with  $h(x)=\mathds{1}_{x\in I}$. 
\end{remark}

\begin{example}	\label{ex:freecompression}
	As an illustration let us consider the spectrum of sub-matrices of Haar distributed orbits. Let $M=UDU^*$ with $D$ diagonal and $U$ Haar distributed unitary. In that case $g_p=\kappa_p(\mu_D)$,  the free cumulants of the spectral measure of $D$ (they are $x$-independent since Haar orbits are unstructured). Then, $b_\ell$ is $x$-independent and $a_\ell(x)=0$ for $x\not\in I$ (of course both $a$ and $b$ depend on $\ell$). Equation \eqref{eq:saddle} then becomes 
\begin{equation}
	A_\ell=\frac{\ell}{z-b_\ell}, \;\; b_\ell = \sum_{k\geq 1} A_\ell^{k-1}\kappa_k(\mu_D),
	\label{eq:saddlepointspecialcase}
\end{equation}
with $A_\ell=\int_{I}\!\mathrm dx\, a_\ell(x)$. 

Recall now that, given a probability measure $\mu$, the generating function of its free cumulants, $$K(z):=\frac{1}{z} + \sum_{k\geq 1}\kappa_k(\mu)z^{k-1}$$ is the inverse of its Cauchy (or Stieltjes)  transform $\mathcal{G}(z)$, i.e. $K(\mathcal{G}(z))=z$.
Further, we define its free compression $\mu^{(\ell)}$ as the probability measure with cumulants $\kappa_k(\mu^{(\ell)}) = \ell^{k-1} \kappa_k(\mu)$ (see \cite[Lecture 14]{Nica2006Lecture} for background on the terminology).

Back to our example, let $K_1$ be the free cumulant generating function of $\mu_D$, i.e. $K_1(z)=\frac{1}{z} +\sum_{k\geq 1}\kappa_k(\mu_D)z^{k-1}$ and let $K_\ell$ be that of its free compression $\mu_D^{(\ell)}$. Then $K_1(z)= K_{\ell}(z/\ell)+(1-\ell)/z$. Similarly let $\mathcal{G}_1$ (resp. $\mathcal{G}_\ell$) be the Cauchy transform of $\mu^D$ (resp. of its free compression).

The equations \eqref{eq:saddlepointspecialcase} can be rewritten, using the definition of free compression,  as  
$$
K_\ell(A_\ell/\ell) = z \text{ that is equivalent to }  \frac{A_\ell}{\ell}= \mathcal{G}_\ell(z).
$$
Since, from Claim \ref{claim:subblock}, the Cauchy transform of the spectral measure of the sub-matrix is $\frac{1}{z-b_\ell}=\frac{A_\ell}{\ell}$, this shows that the Cauchy transform of the sub-matrix is precisely $\mathcal{G}_\ell(z)$. In other terms, the limiting spectral measure of the sub-matrix is $\mu_D^{(\ell)}$,  the free compression of $\mu_D$. 
\end{example}

\subsection{Application: Almost sure classicality of quantum transport} \label{sec:application-self-averaging}
As we have mentioned in \eqref{eq:applicationWick}, the multi-point correlation functions of QSSEP  are given by the determinant of the random matrix $G$. This allows to compute the Laplace transform of the density profile for any realization of the noise as a determinant (see \eqref{eq:rho-on-O})
$$ \Tr\left( \rho_t e^{\sum_{j=1}^N h_j \hat n_j } \right)=\det\left[I + G(e^H-I)\right] , $$ 
where $H$ is the diagonal matrix with diagonal elements $h_i$. By taking the logarithm, we define 
\begin{equation}
\mathcal{F}_Q^N(h) = \frac{1}{N} \log \Tr\left( \rho_t e^{\sum_{j=1}^N h_j \hat n_j } \right) =  \frac{1}{N}  \tr \log \left( I + G(e^H-I) \right) .
	\label{eq:correlationbytracelog}
	\end{equation}
The matrix $\log\left(I + G(e^H-I)\right)$ should be viewed as a formal power series in $G$ (mathematically, one should impose some hypotheses on $H$ to make all series converge). The statement \eqref{eq:PM-self} suggests that $\mathcal{F}_Q^N$ is asymptotically deterministic,  hence equal to its mean.

Furthermore, thanks to the correspondence \eqref{eq:functionalQSSEPSSEP} between the classical SSEP and averaged QSSEP, we have
\begin{equation*}
\mathcal{F}_\mathrm{ssep}(h) = 	\lim_{N\to\infty } \frac{1}{N} \log \mathbb{E} \left[ \Tr\left( \rho_t e^{\sum_{j=1}^N h_j \hat n_j } \right)\right].
\end{equation*}
Hence using \eqref{eq:PM-self} with $P(G)=\log(I+ G(e^H-I))$ -- extrapolating the result from polynomials to power series, which would require a mathematical argument -- we obtain 
\begin{equation*}
\mathcal{F}_\mathrm{ssep}(h) =  \lim_{N\to\infty }  \mathbb E\left[ \mathcal{F}_Q^N(h)  \right].
\end{equation*}
 This limit can be computed via \eqref{eq:correlationbytracelog} using Claim \ref{claim:subblock},  reducing it to a variational problem. (The similarity between the variational problem in Claim \ref{claim:subblock} and that in \eqref{eq:variationalpb} is striking).

Using again that $\mathcal{F}_Q^N(h)$ is asymptotically deterministic, we learn 
\begin{equation}
 \lim_{N\to\infty } \mathcal{F}_Q^N(h) = \mathcal{F}_\mathrm{ssep}(h).
 \label{eq:classicality}
\end{equation}
This is one form of the \emph{classicality} of quantum transport in large diffusive systems, as  mentioned in Remark \ref{remark:classicality}. This argument is not yet fully rigorous: the steps described above need to be better justified. But the alternative combinatorial approach discussed below yields strong hints of the validity of \eqref{eq:classicality}, hence indirectly that the matrix $G$ of stationary QSSEP satisfies property \textit{(4)} of Assumptions \ref{assumption}.

\subsection{Back to classical SSEP}
Since the occupation numbers in SSEP are Bernoulli variables, $n_j^2=n_j$, their cumulant generating function
$$\log \mathbf E_{\rm ssep}\left[ e^{\sum_{j=1}^N h_j n_j} \right] $$
is fully determined by the multiple cumulants at non-coinciding points, that is, by the set of cumulants $\mathbf{K}(n_{j_1},\cdots,n_{j_p})$ with all indices $j_k$ distinct. Expressing this generating function in terms of those cumulants is an intricate combinatorial problem, solved in  \cite{bauer2022bernoulli}. If the multiple cumulants at non-coinciding points scale appropriately with $N$, its solution takes a variational form in the large $N$ limit. Furthermore, the correspondence \eqref{eq:functionalQSSEPSSEP} between the classical SSEP and QSSEP, in the form \eqref{eq:applicationWick}, allows to express these cumulants in terms of the QSSEP local free cumulants. This leads to the variational problem \eqref{eq:variationalpb} connecting $\mathcal{F}_\mathrm{ssep}$ to free probability as announced in the Introduction.

\section{Mathematical open problems}
\label{sec:open}

We conclude with a summary of open problems for mathematicians raised by the results discussed above. 
\subsection{Properties of the invariant measure}
\label{sec:openstationarymeasure}
Prove that the stationary measure of the open QSSEP is unique among probability measures on density matrices that are invariant under the global $U(1)$ symmetry $\rho\mapsto e^{\I\theta\hat N}\rho\,e^{-\I\theta\hat N}$, $\hat N=\sum_j c_j^\dagger c_j$ (i.e. under changing all $c_j$ into $e^{\I\theta}c_j$), and that it is supported on Gaussian states, that is states of the form \eqref{eq:rho-quadratic}.

\subsection{Mixing times} Study the mixing time of the diffusion $G_{t+dt} = e^{-\I dh_t} G_t e^{\I dh_t}$ on unitary orbits and compare to Brownian motion on $U(N)$ -- see Remark \ref{rem:mixing}. 

\subsection{Scaling limit of QSSEP cyclic cumulants} Prove Claim \ref{claim:freecumulantsg}. There are two possible routes. The first one would be to prove the recursive characterization of the functions $g_p$ found in \cite{bernard2021solution}. Then, Claim \ref{claim:freecumulantsg} would be implied by \cite[Theorem 6.1]{biane2021combinatorics}. Alternatively, one could prove that the stationary $G$ matrix of QSSEP satisfies properties of an appropriate ensemble of structured matrices (in particular proving item (4) of Assumptions \ref{assumption}) satisfying Claim \ref{claim:traces}.

\subsection{Ensembles of structured matrices} 
\label{sec:openstructuredmatrices}
Prove that the ensemble of structured matrices defined by Assumptions \ref{assumption} 
is stable under transformations considered in Claim \ref{claim:polynomialstability} and Claim \ref{claim:entrywisetransformstability}. 

\subsection{Spectrum of structured random matrices} Prove Claim \ref{claim:subblock}, which characterizes the limiting spectral distribution of general ensembles of structured random matrices.  It would also be interesting to study if the saddle point corresponds to a global  extremum of the functional $\mathcal F(h; z)$.

\subsection{Almost-sure classicality} Establish the classicality of quantum transport for the QSSEP, that is, prove the almost sure convergence \eqref{eq:classicality}.

	\subsection*{Acknowledgments}
	These notes were written on the occasion of a series of lectures given by D.B. in the spring of 2025 at the Center for Mathematical Sciences and their Applications, Harvard University. We thank the CMSA for hospitality. G. Barraquand was supported by ANR grants ANR-21-CE40-0019 and ANR-23-ERCB-0007. D. Bernard was supported via ANR project ESQuisses under contract number ANR-20-CE47-0014-01.

	\printbibliography
	
	
	\end{document}